\documentclass[11pt]{article}

\usepackage{tikz-cd}
\usepackage{amssymb}
\usepackage[T1]{fontenc}
\usepackage[utf8]{inputenc}
\usepackage{babel}
\usetikzlibrary{babel}
\usepackage{booktabs}
\usepackage{xcolor}

\usepackage{mathrsfs}
\usepackage{amsfonts,amsmath,lmodern,fancyhdr,lastpage,graphicx,nccfoots,caption}
\usepackage{afterpage}

\usepackage{blkarray, bigstrut}
\usepackage{tikz}
\usepackage{capt-of}
\usepackage{graphicx}
\usepackage{amssymb}
\usepackage{amsmath}
\usepackage{amsfonts}
\usepackage{amsthm}
\usepackage{amscd}
\usepackage{epstopdf}
\usepackage{mathtools}
\usepackage{color}
\newtheorem{thm}{Theorem}[section]

\newtheorem{cor}[thm]{Corollary}

\newcommand{\ZZ}{\mathbb{Z}}

\newcommand{\QQ}{\mathbb{Q}}
\newcommand{\RR}{\mathbb{R}}
\newcommand{\CC}{\mathbb{C}}
\newcommand{\PP}{\mathbb{P}}
\newcommand{\tF}{\tilde{\mathcal{F}}}
\newcommand{\tUx}{\tilde{U}_x}

\newcommand{\hooklongrightarrow}{\lhook\joinrel\longrightarrow}
\DeclarePairedDelimiter\ceil{\lceil}{\rceil}
\DeclarePairedDelimiter\floor{\lfloor}{\rfloor}

\newenvironment{example}
{ \begin{flushleft} \textbf{Example.}}
{ \end{flushleft} }

\newenvironment{remark}
{ \begin{flushleft} \textbf{Remark:}}
{ \end{flushleft} }

\newtheorem{lemma}{Lemma}

\newtheorem{theorem}{Theorem}

\newtheorem{corollary}{Corollary}

\newtheorem{proposition}{Proposition}

\usepackage{hyperref}

\vfuzz \hfuzz
\newcounter{a}
\ifodd\thea\else\stepcounter{a}\fi

\begin{document}
\thispagestyle{plain}
\begin{center}
\Large
\textsc{Surface quotient singularities and bigness of the cotangent bundle: Part II}
\end{center}

\begin{center}
\textit{Yohannes D. Asega}, \textit{Bruno De Oliveira}, \textit{Michael Weiss}  \smallskip \\
\begin{tabular}{l}
\small University of Miami \\
\small 1365 Memorial Dr                    \\
\small Miami, FL 33134       \\
\small e-mail: \texttt{y.asega@math.miami.edu}\\
 \small \phantom{e-mail: }\texttt{bdeolive@math.miami.edu}  \\ 
\small \phantom{e-mail: }\texttt{weiss@math.miami.edu}

\end{tabular}
\end{center}

\noindent
 
\medskip 
\begin{abstract}

In two parts, we present a bigness criterion for the cotangent bundle of resolutions of orbifold surfaces of general type. As a corollary, we obtain  the \textit{canonical model singularities} (CMS) criterion  that can be applied to determine when a birational class of surfaces has smooth representatives with big cotangent bundle and compare it with other known criteria. We then apply this criterion to the problem of finding the minimal degrees $d$ for which the deformation equivalence class of a smooth hypersurface of degree d in $\PP^3$ has a representative with big cotangent bundle; applications to the minimal resolutions of cyclic covers of
$\PP^2$ branched along line arrangements in general position are also obtained. 

The CMS criterion involves invariants of canonical singularities whose values were unknown. In this part of the work, we describe a method of finding these invariants and obtain formulas for $A_n$-singularities. We also use our approach to derive several extension results for symmetric differentials on the complement of the exceptional locus $E$ of the minimal resolution of an $A_n$-singularity; namely, we characterize the precise extent to which the poles along $E$ of symmetric differentials on the complement are milder than logarithmic poles.
\end{abstract}

\noindent\textbf{keywords:} $A_n$ singularities; symmetric differentials;
big cotangent bundle; extension results; surfaces of general type.

\section{Introduction}
\label{sec:a}

\ 

\

In Part I of this work \cite{ADOWI}, we present the (quotient singularities) QS-bigness criterion and its special case, the (canonical model singularities) CMS-criterion, concerning bigness of the cotangent bundle of surfaces of general type $X$. The CMS-criterion for the bigness of the cotangent bundle of $X$ has two terms: one comes from the topology of the minimal model $X_{\text{min}}$ of $X$ and the other comes  from the contribution of the singularities in the canonical model $X_{\text{can}}$ of $X$ to the $m$-asymptotic growth of $h^1(X_{\text{min}},S^m\Omega^1_{X_{\text{min}}})$. The main purpose of this paper is to determine this contribution when the singularities are of type $A_n$. Our approach also leads to several extension results for symmetric differentials on the complement of the exceptional locus of the minimal resolution of an  $A_n$ singularity.

\ 

In Section 1, we give the background and define the strategy to determine  the $1$st cohomological $\Omega$-asymptotics, ${h}_\Omega^1(y)$,  of a log terminal (quotient) surface singularity $y$. These are the surface invariants involved in the QS-bigness criterion. Our approach uses the theory developed by Wahl \cite{wahl_chernclasses}, Blache \cite{blache} and Langer \cite{langer} for  the Chern classes (local and global) and the asymptotic Riemann-Roch formulas for orbifold vector bundles.  

\ 

The $1$st cohomological $\Omega$-asymptotics, ${h}_\Omega^1(y)$, of a surface singularity $y$  is defined by

$${h}_\Omega^1(y)=\liminf_{m\to \infty}\frac{h^0(U_y,R^1\sigma_*S^m\Omega^1_{U_y})}{m^3},$$ 

\noindent where $U_y$  is a neighborhood germ of the surface singularity $y$ and  $\sigma:\tilde {U}_y\to U_y$ its  minimal resolution. It was shown in the proof of the QS-criterion (Theorem 2 of Part I) that if $X$ is the minimal resolution of a surface of general type $Y$ with quotient singularities, then

$$\sum_{y \in \operatorname{Sing}(Y)} {h}_\Omega^1(y)\le h^1_\Omega(X):=\lim_{m\to \infty}\frac {h^1(X,S^m\Omega^1_{X})}{m^3}$$

The right side $h^1_\Omega(X)$ is the measurement of the $m$-asymptotic growth of $h^1(X,S^m\Omega^1_{X})$. The sum on the left side is called \textit{the localized  component  of  
 $h^1_\Omega(X)$} and we denote it by $Lh^1_\Omega(X)$. It is a lower bound for $h^1_\Omega(X)$ determined by the singularities of $Y$.

\ 

 Our approach to find ${h}_\Omega^1(y)$ uses the relations between the local invariants appearing in the comparison  between the  Euler characteristics an orbifold vector bundle $V$ on an orbifold surface $Y$ and of a vector bundle $\tilde V$  on its minimal resolution $X$, $\sigma:X\to Y$, with $V=\sigma_*\tilde V$. The key relation for us, requires $m$-asymptotic results  for $V=(S^m\Omega^1_Y)^{\vee\vee}$  (see section 3.1), is:

 $$\lim_{m\to \infty} \frac{\hslash^0(y,m)+h^1(y,m)}{m^3}=\frac{1}{3!}\left((c_2(y,T_{{X}})-c_1^2(y,T_{{X}})\right)$$

\noindent where $h^1(y,m):=h^0(U_y,R^1\sigma_*S^m\Omega^1_{U_y})$, $c_2(y,T_{{X}})$, $c_1^2(y,T_{{X}})\in \QQ$ are the local Chern numbers of $y$ and 

$$\hslash^0(y,m)=\dim [H^0(\tilde{U_y}\setminus E, S^m\Omega_{\tilde{U_y}}^1)/H^0(\tilde{U_y},S^m \Omega_{\tilde{U_y}}^1)],$$

\noindent where $(\tilde U_y,E)$ is the minimal resolution of the neighborhood germ $(U_y,y)$.
 
\

In Section 2, we present a method to find the invariants $\hslash^0(y,m)$ and their asymptotics when $y$ is an $A_n$ singularity. We show that for fixed $n$, $\hslash^0(A_n,m)$ is a polynomial of degree 3 in $m$ up to the linear term. In fact, it is a quasi-polynomial in $m$, i.e. coefficients are periodic functions of $m$, with the cubic and quadratic coefficients constant. The main interest in this work, with respect to the discussion above, lies in finding the cubic term of such a polynomial. In \hyperref[thm:1]{Theorem 1}, we give a closed formula for the cubic and quadratic coefficients of $\hslash^0(A_n,m) $. We note that the special case of $A_1$ was considered in \cite{bogomolov_nodes}, see \cite{thomas} for the cubic coefficient and \cite{bruin2022explicit} for the quasi-polynomial.

\

\begin{theorem}\label{thm:1}  Let $(\tilde X,E)$ be  the minimal resolution of the germ   $(X,x)$ of an  $A_n$ singularity. Then the following holds for $\hslash^0(A_n,m)=\dim [H^0(\tilde {X}\setminus E,S^m\Omega^1_{\tilde {X}}) / H^0(\tilde {X},S^m\Omega^1_{\tilde {X}})]$:

\begin{itemize}

\item  [(a)] $\hslash^0(A_n,m)$ is given by a weighted lattice sum over a polygon $\mathcal{P}_n(m)$:

$$\hslash^0(A_n,m)=\sum_{\substack{  {\bf x}=(x_1,x_2)\in \mathcal{P}_n(m)\cap \ZZ^2\\x_1+(n+1) x_2\equiv m  \bmod{2}}}h_{n,m}({\bf x}) $$
\hspace {-0.05in}
\noindent where

\begin{itemize}

    \item [i)] the polygon $\mathcal{P}_n(m)$ is symmetric about the $x_2$-axis and $\mathcal{P}_n(m)\cap \{x_2\ge 0\}$ is given by the inequalities $-x_1\le 0$, $-x_2\le 0$, $x_1-(n-1)x_2\le m$ and $-x_1+(n+1)x_2\le m+2$.

    \item [ii)] the weight function $h_{n,m}({\bf x})=\min \left\{\displaystyle\sum_{r=0}^{n-1} \alpha_{n,m,r}({\bf x}),\beta_{n,m}({\bf x})\right\}$, with 
    \begin{align*}
    \alpha_{n,m,r}({\bf x})&=\max\left\{0,\frac{m-x_1+(2r-n+1)x_2}{2}\right\}\\
    \beta_{n,m}({\bf x})&=\max\left\{0,m+1-\alpha_{n,m,-1}({\bf x})-\alpha_{n,m,n}({\bf x})\right\}.
    \end{align*} 
\end{itemize}

 \item [(b)] $\hslash^0(A_n,m)=\hslash^0_\Omega(A_n)m^3+3\hslash^0_\Omega(A_n)m^2+O(m)$, with

 \begin{align*} \hslash^0_\Omega(A_n):=\lim_{m\to \infty}\frac{\hslash^0(A_n,m)}{m^3}=\frac {4}{3}\sum^n_{j=1}\frac {1}{j^2}-\frac{12n^4+65n^3+117n^2+72n}{6(n+1)^2(n+2)^2}.\end{align*}
\end{itemize}
\end{theorem}

\ 

\ 

To understand the weight function, some setup needs to be introduced. Let $\tilde X$ be  minimal resolution of the $A_n$ affine model $X=\{xz+y^{n+1}=0\}\subset \CC^3$ and $E$ the exceptional locus. Using the smoothing $\pi:\CC^2\to X$ we define a 3-gradation on the algebra $S(\tilde X\setminus E):=\bigoplus_{m=0}^\infty H^0(\tilde X\setminus E,S^m\Omega^1_{\tilde X})$. A differential $w$ has 3-degree $(\hat k,i,m)$ if $w\in H^0(\tilde {X}\setminus E,S^m\Omega^1_{\tilde {X}})$ and when $w$ is viewed in $H^0(\CC^2,S^m\Omega^1_{\CC^2})^{\ZZ_{n+1}}$ it vanishes to order $i$ at $0$, for the degree $\hat k$ it comes from a block partition on the $\ZZ_{n+1}-$ invariant differential monomials, see \hyperref[sec:block-partition]{Section 2.2.1}.

\

The weight function $h_{n,m}({\bf x})$ gives the dimension of the obstruction space for differentials in the $(x_2,x_1,m)$-graded piece of $S(\tilde X\setminus E)$  to extend regularly along $E$. The polygon $\mathcal{P}_n(m)$ lives on the $(x_1,x_2)$-plane, $x_1=i$ and $x_2=\hat k$. The weight function naturally induces a decomposition of the polygon  $\mathcal{P}_n(m)=\cup_{\ell=1}^{4n}\mathcal{P}_n^{\ell}(m)$ by convex polygons where the weight functions $h_{n,m,\ell}({\bf x}):=h_{n,m}({\bf x})|_{\mathcal{P}_n^{\ell}(m)}$ are defined by a single polynomial of degree 1 in $x_1$, $x_2$ (and $m$). 

\ 

The function $\hslash^0(A_n,m)$ being a quasi-polynomial of degree 3 in $m$ follows from the theory of polynomial weighted lattice sums over convex polytopes $\mathcal{P}({\bf b})$, where the parameter ${\bf b}=(b_1,...,b_k)$  defines $\mathcal{P}({\bf b})$ via the inequalities, $\mu_l({\bf x} )\le b_l$, where the linear  forms $\mu_l({\bf x})$ are fixed, but ${\bf b}$ varies (see \cite{Ehrhart}, \cite{Brion}, \cite{Vergne}). Each polygon 
$\mathcal{P}_n^{\ell}(m)$ is of the form $\mathcal{P}({\bf b}(m))$ with ${\bf b}(m)=(\alpha_{\ell,1}m+\beta_{\ell,1},...,\alpha_{\ell,k}m+\beta_{\ell,k}$ with $k=3,4$.  Relevant to the task of finding the period of the quasi-polynomial $\hslash^0(A_n,m)$ (i.e. the lcm of the periods of the coefficients of the quasi-polynomial) is that $\alpha_{\ell,i},\beta_{\ell,i}$ and the coefficients of the fixed linear forms $\mu_l({\bf x})$ are in $\QQ$ (and we know the denominators).

\ 
 
 On future work we will describe non-asymptotic features of the quasi-polynomial $\hslash^0(A_n,m)$. One such feature is a divisibility condition  for the period. This result allows us to determine  the quasi-polynomials $\hslash^0(A_n,m)$ for low $n$ and use them to investigate the presence of symmetric differentials of low degrees on the resolution of certain classes  surfaces (e.g. hypersurfaces in $\PP^3$) with $A_n$ type singularities.

\

In Section 3, we derive applications of our knowledge of $\hslash^0(A_n,m)$ and the theory developed in Section 2 describing the extension properties of differentials on $\tilde X\setminus E$. The first application is the main purpose of Part II of this work, that is, finding the closed formula for the $1$st cohomological $\Omega$-asymptotics ${h}_\Omega^1(x)$ for $A_n$ singularities. 

\smallskip 

\begin{theorem}\label{thm:2} The 1st cohomological $\Omega$-asymptotics of an $A_n$ singularity is given by:

\begin{align*} h^1_\Omega(A_n)=\frac{n^5+19n^4+83n^3+137n^2+80n}{6(n+1)^2(n+2)^2}-\frac {4}{3}\sum^n_{k=1}\frac {1}{k^2}, \end{align*}

\end{theorem} 

\  

In Part I \cite{ADOWI} Section 3, Theorem 2 is used in combination with the CMS-bigness criterion to obtain the strongest known results  concerning:

\smallskip 

i) the minimum  $d_{\min}$ of the degrees $d$ for  which the deformation equivalence class of a smooth hypersurfaces of  $\PP^3$ of degree $d$ has  representatives with big cotangent bundle. Theorem 4 of Part I shows $d_{\min}\le 8$, moreover in theory with Theorem 2 and considering resolutions of hypersurfaces with only $A_n$ singularities one could achieve $d_{\min}\le 6$, but never $d_{\min}=5$;

\smallskip 

ii) bigness of the cotangent bundle of the resolutions of cyclic coverings of $\PP^2$ branched along line arrangements, see Theorem 5 of Part I.

\

\

The second application is concerned with extension results and characterizes: 1) the precise extent to which the poles along $E$ of symmetric differentials on the complement of the exceptional locus $E$ of the minimal resolution of an  $A_n$ singularity are milder than logarithmic poles (\cite{miyaoka1984maximal} 4.14, see also \cite{wahl_chernclasses} 4.7 showed that the poles are at most logarithmic);  2) a sufficient condition for a symmetric differential  on the complement of the exceptional locus $E$ of the minimal resolution of the germ of an  $A_n$ singularity to extend holomorphically through $E$.

\

\begin{theorem}\label{thm:3} Let $\sigma:(\tilde X,E)\to (X,x)$ and $\pi:(\CC^2,0)\to (X,x)$ be the minimal resolution and the smoothing of  the $A_n$ singularity germ  $(X,x)$.  Then:

\begin{itemize}
    \item [(a)] The maximal divisor $D$ such that:
\vspace {-.05in}
\begin{align*}H^0(\tilde {X}\setminus E,S^m\Omega^1_{\tilde {X}})= H^0(\tilde {X},S^m\Omega^1_{\tilde {X}}(\log E)\otimes \mathcal {O}_{\tilde X}(-D))  \end{align*}
is given by
\vspace {-.18in}

\begin{align*} D =\sum_{r=1}^{n}\left(\sum_{j=0}^{\min(r-1,n-r)}\ceil{\frac{m-2j}{n+1}}\right)E_r
         \end{align*}

         \noindent with $E=U_{r=1}^nE_r$ the  exceptional locus (ordering such  that $E_r\cap E_{r+1}\neq\emptyset$).

         \item [(b)] Let $w\in H^0(X\setminus {x},S^m\Omega^1_X)$ and $\bar w\in H^0(\CC^2,S^m\Omega^1_{\CC^2})$ with $\bar w|_{\CC^2\setminus 0}=\pi^*w$. Set $ord(w):=\max_{i}\{\bar w \in \mathfrak{m}^i S^m\Omega^1_{\CC^2}  \}$, $\mathfrak{m}$ the maximal ideal at $0$. 
         
         \smallskip  
         
         Then $\sigma^*w$ extends holomorphically to $\tilde X$, if ord$(w) \ge nm$.
\end{itemize}
\end{theorem}

\ 

The second author would like to thank Fedor Bogomolov and Margaridade Mendes Lopes for comments on this work and the faculty and staff of the Center for Mathematical Analysis, Geometry and Dynamical Systems of the University of Lisbon for the support on his stay at Center.

\newpage

\section{Background}
\label{sec:a}

\ 

\

Let $Y$ be an orbifold surface (i.e. with only quotient singularities) of general type and $\sigma:X\to Y$ its minimal resolution. In this section we succinctly describe the invariants of orbifold surface singularities involved in determining the presence of symmetric differentials on $X$ with an emphasis on the asymptotics $\lim_{m\to \infty}\frac{h^0(X,S^m\Omega^1_X)}{m^3}$ (for more details see Section 1 of Part I of this work \cite{ADOWI}).

\ 

Riemann-Roch gives:

\begin{align*}
    h^0(X,S^m\Omega^1_X)=\int_X \text{ch}({S^m\Omega^1_X})\text{td}(X)+h^1(X,S^m\Omega^1_X)-h^2(X,S^m\Omega^1_X) \tag {1.1}
\end{align*}

\ 

The first term $\int_X \text{ch}({S^m\Omega^1_X})\text{td}(X)$ is a polynomial of degree 3 in $m$ with the coefficients involving the Chern classes of $X$. The term $h^2(X,S^m\Omega^1_X)$ vanishes for $m\ge 3$  due to Bogomolov's vanishing, \cite {bogomolov_stability}. The term $h^1(X,S^m\Omega^1_X)$ has the following decomposition (for further details see proof of Theorem 2 in Part I \cite{ADOWI}):

\begin{align*}
    h^1(X,S^m\Omega^1_X)=Lh^1(X,S^m\Omega^1_X)+NLh^1(X,S^m\Omega^1_X) \tag {1.2}
\end{align*}

\noindent where the localized component (at the singularities):

\begin{align*}
  Lh^1(X,S^m\Omega^1_X):=\sum_{y \in \text{Sing}(Y)}h^1(y,S^m\Omega^1_X) \tag {1.3}  
\end{align*}

\noindent with $h^1(y,S^m\Omega^1_X):=h^0(U_y,R^1\sigma_*S^m\Omega^1_{\tilde X})$, where $U_y$ is an affine neighborhood with $U_y\cap \text{Sing}(Y)=\{y\}$ and the non-localized component $NLh^1(X,S^m\Omega^1_X):=h^1(Y,\sigma_*S^m\Omega^1_X)$.

\ 

The QS-bigness criterion (Theorem 2 of Part I) gives the criterion for the bigness of the cotangent bundle of $X$ ($\lim_{m\to \infty}\frac{h^0 (X,S^m\Omega^1_X)}{m^3}\neq 0$):

\begin{align*}
     \sum_{y \in \operatorname{Sing}(Y)} \liminf_{m\to \infty}\frac{h^1(y,S^m\Omega^1_X)}{m^3} + \frac{s_2(X)}{3!} >0 \implies \Omega^1_X \text {  big} \tag{1.4}
\end{align*}

\noindent where $\frac{s_2(X)}{3!} = \frac{c^2_1(X)-c_2(X)}{6}$ is the cubic coefficient and the leading term in $m$ of $\int_X \text{ch}({S^m\Omega^1_X})\text{td}(X)$.

\ 

The main goal of this work is to provide a method to find $h^1(y,S^m\Omega^1_X)$ and determine its asymptotics, $\lim_{m\to \infty}\frac{h^1(y,S^m\Omega^1_X)}{m^3}$,  when $y$ is an $A_n$ singularity (we will see that this limit exists). We use an indirect approach coming from the theory developed by Wahl \cite{wahl_chernclasses}, Blache \cite{blache} and Langer \cite{langer} concerning orbifold vector bundles, their Chern classes, and Riemann-Roch formulas (with their asymptotics). 

\

\noindent \textbf{Notation}. Hereon,  $h^1(y,m):=h^1(y,S^m\Omega^1_X)$ and $h^1(A_n,m):=h^1(y,m)$ where $y$ is the singularity $A_n$.

\

For further background on the general theory concerning what follows, see Section 1 of Part I \cite {ADOWI} (or see \cite{blache}, \cite{langer}, \cite{wahl_chernclasses}).

\

The difference between the Euler charateristics of the vector bundle $S^m\Omega^1_X$ on $X$ and  of the orbifold vector bundle $\hat S^m\Omega^1_Y:=(\sigma_*\Omega^1_X)^{\vee\vee}$ on $Y$ is measured by the sum over all the singular points $y$ of $Y$ of the local Euler characteristic of $\hat S^m\Omega^1_Y$ at $y$:

\begin{align*}
    \chi (y, m):=\chi (y, S^m\Omega^1_X) := \hslash^0(y,m)+h^1(y,m) \tag {1.5}
\end{align*}

\begin{align*}
			\hslash^0(y,m)&:=\dim [H^0(\tilde{U_y}\setminus E, S^m\Omega_X^1)/H^0(\tilde{U_y},S^m \Omega_{{X}}^1)] \tag {1.6}
		\end{align*}

\ 

The relations between the Euler characteristic of $\hat S^m\Omega^1_Y$ on $Y$, the orbifold Euler characteristic of $\hat S^m\Omega^1_Y$ on $Y$, and the orbifold Euler characteristic of $S^m\Omega^1_X$ on $X$, where the orbifold Euler characteristic of an orbifold vector bundle ${\mathcal{F}}$ on a orbifold $Z$ is $\chi_{\operatorname{orb}}(Z, {\mathcal{F}}):=\int_Z \text{ch}_{\operatorname{orb}}({\mathcal{F}})\text{td}_{\operatorname{orb}}(Z)$ (involving orbifold Chern classes) give:

\begin{equation} h^1(y,m)=
			\mu (y, m)- \chi_{\operatorname{orb}} (y, m)-\hslash^0(y,m)\tag {1.7}
		\end{equation}

  \ 

  \noindent with:

  \begin{align}\mu(y,m):= \frac{1}{|G_y|} \sum_{g \in G_y\setminus\{\text{Id}\}} \frac{\text{Tr}(\rho_{\hat S^m\Omega_{{Y}}^1}(g))}{\det (\text{Id} - g)} \tag {1.8}
		\end{align}

	\begin{equation}
			\chi_\text{orb} (y,m) := \frac{s_2(y)}{3!} m^3 - \frac {1}{2}c_2(y)m^2-\frac {c_1^2(y)+3c_2(y)}{12}m+\frac {c_1^2(y)+c_2(y)}{12}\tag{1.9}
		\end{equation}

	\noindent where:

 \
 
 \noindent i) $\mu(y,m)$ is the contribution that the singularity $y$ gives to the discrepancy  between the Euler characteristic of $\hat S^m\Omega^1_Y$ and the orbifold Euler characteristic of $\hat S^m\Omega^1_Y$ on $Y$. $G_y\subset GL(2,\mathbb {C})$ is the local fundamental group and $\rho_{\hat S^m\Omega_{{Y}}^1}$ the representation of $G_y$ associated to the orbifold vector bundle $\hat S^m\Omega_{{Y}}^1$ (see \cite{blache} 2.6 for the bijective association of isomorphism classes of representations of $G_y$ to isomorphism classes of germs of orbifold vector bundles  at the quotient singularity with local fundamental group $G_y$).

 \
 
 \noindent ii)  $\chi_\text{orb} (y,m)$ is the contribution that the singularity $y$ gives to the discrepancy  between the orbifold Euler characteristics of $S^m\Omega^1_X$ on $X$ and  of $\hat S^m\Omega^1_Y$ on $Y$.
	$$c^2_1 (y):=c_1^2(y,T_{{X}})=c_1^2(y, \Omega_{{X}}^1)\in \QQ,$$  $$c_2(y):=c_2(y,T_{{X}})=c_2(y, \Omega_{{X}}^1)\in \QQ,$$ $$s_2(y)=c_1^2(y)-c_2(y)$$

\noindent are local Chern numbers at the singularity $y$. They are obtained from the local  Chern classes $c_i(y,T_X)\in H^{2i}_{\text{dRc}}((\tilde U_y,E),\QQ)$, $\tilde U_y=\sigma^{-1}(U_y)\subset X$ and dRc stands for de Rham cohomology with compact support (\cite[\S3]{blache}).

 \newpage

\section{The invariants $\hslash^0(y,m)$ and their asymptotics for $A_n$ singularities}

\

\subsection{ $A_n$ model}

\ 

The affine model for the $A_n$ singularity that will be used is $X=\{xz-y^{n+1}=0\} \subset \mathbb{C}^3$. The affine surface $X$ is the quotient space obtained from $\mathbb{C}^2$ via the diagonal action of $\mathbb {Z}_{n+1}$ coming from the representation $\rho: \mathbb {Z}_{n+1}=\langle\tau\rangle \to SL(2,\mathbb {C})$ with:

\begin{align}
\rho(\tau)=\left [
\begin{array}{cl} \epsilon&0\\0&\epsilon^n\end{array}\right ] \tag {2.1}    
\end{align}

\smallskip

\noindent where $\epsilon$ is a $n+1$-primitive root of unity.

\ 

The standard smoothing of ${X}$, $\pi:\mathbb{C}^2 \to X$, is given by $\pi(z_1,z_2)=(z_1^{n+1},z_1z_2,z_2^{n+1})$. Let $\sigma:\tilde X \to X$ be the minimal good resolution of $X$, which can be obtained via successive blow ups of the ambient space $\mathbb{C}^3$ at points infinitesimally near the origin ($n$ blow ups for $n$ odd and $n-1$ for $n$ even). Let $\varphi:\mathbb{C}^2 \dashrightarrow \tilde X$ be the $n+1$ to 1 rational map $\sigma^{-1}\circ \pi$, whose indeterminacy locus is $\text {Ind}(\varphi)=\{(0,0)\}$. 

\

\ 

 The resolution $\tilde{X}$ has a covering consisting of $n+1$ open sets $U_r$, $r=0,...,n$, isomorphic to the affine plane. The  isomorphisms $\phi_r:\CC^2 \to U_r$,   with $u_{1}$ and $u_{2}$ coordinaters for $\CC^2$, can be chosen such that the following diagram holds: 

	\begin{center}
		\begin{tikzpicture}
			\begin{scope}
				\node (tC) at (-6.7, 1.5) {$\CC^2$};
				\node (C) at (-12.8,0) {$\CC^2 $};
				\node (U) at (-11,0) { $U_r\subset$};
				\node (tX) at (-10.3,0) {$\tilde{X}$};
				\node (X) at (-6.7, 0) {$X = \{xz - y^{n+1} = 0\} \subset \CC^3$};
				
				\draw[->, to path={-| (\tikztotarget)}] (C)  -- node[above] {$\phi_r$}   (U);
				\draw[->, to path={-| (\tikztotarget)}] (tX) -- node[above] {$\sigma$} (X);
				\draw[->, to path={-| (\tikztotarget)}] (tC) -- node[right] {$\pi$,$(z_1^{n+1},z_1z_2,z_2^{n+1})$} (X);
				\draw[dashed, ->,  to path={-| (\tikztotarget)}] (tC) to [out=160, in = 45] node[above left] {$\varphi_r$} (C);
				\draw[dashed, ->,  to path={-| (\tikztotarget)}] (tC) to [out=200, in = 40] node[above left] {$\varphi$} (tX);
			\end{scope}
		\end{tikzpicture}
	\end{center}

\noindent 	and
	
		\begin{align*}
			\varphi_r^*u_{1}=z_1^{n+1-r}z_2^{-r} \hspace {.6in} \varphi_r^*u_{2}=z_1^{r-n}z_2^{r+1}\tag {2.2}
		\end{align*}
		
		\

\noindent The exceptional locus of $\sigma$, $E=E_1+...+E_n$, is a sum of $n$ $(-2)$-curves that intersect transversally with intersection properties given by the $A_n$-Dykin diagram. The relation between the open covering $\{U_r\}$ with coordinates $u_{i,r}=u_i\circ \phi_r^{-1}$, $i=1,2$, and the exceptional set is given by:

\vspace {-.1in}

$$ \hspace {.4in} E_j\subset U_{j-1}\cup U_j \hspace {1in} E_j\cap U_r=\emptyset \hspace {.1in}\text {if} \hspace {.1in}r\not=j-1,j$$

\vspace {-.2in}

\begin{align*}E_j\cap U_{j-1}=\{u_{1,j-1}=0\} \hspace {.5in} E_j\cap U_{j}=\{u_{2,j}=0\} \tag {2.3}\end{align*}

\smallskip

Also relevant is the extension of the exceptional set $\hat E=E\cup E_0\cup E_{n+1}$, where $E_0=\{u_{2,0}=0\}\subset U_0$ and $E_{n+1}=\{u_{1,n}=0\}\subset U_n$. The following holds for all $r=0,...,n$:

\vspace {-.1in}

\begin{align*} \tilde X\setminus \hat E=U_r\setminus \{u_{1,r}u_{2,r}=0\} \tag {2.4}\end{align*}

\

\noindent {\bf Note}: To obtain uniformity in the formulas ahead we also consider $r=-1$ and $r=n+1$, where we define $\varphi_r:\CC^2 \to \CC^2$ for $r=-1,n+1$ using the same formulas as in (2.2).

\

The following diagram coming, via restrictions, from diagram (2.2) is also relevant:

\begin{center}
		\begin{tikzpicture}
			\begin{scope}
				\node (tC) at (-0.7, 2) {$\CC^*\times \CC^*$};
				\node (C) at (-4.7,0) {$\CC^*\times \CC^*$};
				\node (X) at (-0.7, 0) {$X^* := (\CC^*\times \CC^*)/\mathbb{Z}_{n+1}$};\node (XXX) at (-8.3, 0) {};
				\node (XX) at (3.85, 0) {(2.5)};
				\draw[->, to path={-| (\tikztotarget)}] (C)  -- node[above] {$\sigma \circ \phi_r$}   (X);
				\draw[->, to path={-| (\tikztotarget)}] (tC) -- node[right] {$\pi$} (X);
				\draw[->,  to path={-| (\tikztotarget)}] (tC) -- node[above] {$\varphi_r$} (C);
			\end{scope} 
		\end{tikzpicture}
	\end{center}

\newpage

\subsection{The 3-gradation of the algebra of symmetric differentials $S(\CC^*\times \CC^*)$  }

\

\subsubsection{Block partition}\label{sec:block-partition}

\

\

The algebra of regular symmetric differentials on $\CC^*\times \CC^*\subset \CC^2$:

\vspace {-.1in}

$$S(\CC^*\times \CC^*)=\bigoplus_{m=0}^\infty H^0(\CC^*\times \CC^*,S^m\Omega_{\CC^2}^1)=\mathbb{C}[z_1,z_1^{-1},z_2,z_2^{-1},dz_1,dz_2]$$

\noindent has the natural bi-gradation with bi-graded pieces:

\vspace {-.1in}

$$S(\CC^*\times \CC^*)_{(i,m)}=H^0(\CC^*\times \CC^*,S^m\Omega^1_{\CC^2})_{(i)}:=\mathbb{C}[z_1,z_1^{-1},z_2,z_2^{-1}]_{(i)}\otimes\mathbb{C}[dz_1,dz_2]_{(m)}$$

\noindent with $i\in \mathbb{Z}$ and $m\in \mathbb{Z}_{\ge 0}$. Also relevant to us is the subalgebra $S(\CC^2)=\mathbb{C}[z_1,z_2,dz_1,dz_2]$ of $S(\CC^*\times \CC^*)$ with induced bi-gradation with bi-graded pieces $H^0( \CC^2,S^m\Omega_{\CC^2}^1)_{(i,m)}$, $i,m\in \mathbb{Z}_{\ge 0}$.

\

A symmetric differential $w\in S(\CC^*\times \CC^*)_{(i,m)}$ is said to be bi-homogeneous of $order$ $i$ and $degree$ $m$. The symmetric differentials $z_1^{i_1}z_2^{i_2}(dz_1)^{m_1}(dz_2)^{m_2}$ will be  called monomials. We have:

\vspace {-.1in}

\begin{align*}
z_1^{i_1}z_2^{i_2}(dz_1)^{m_1}(dz_2)^{m_2} \text { 
 is } \mathbb{Z}_{n+1}\text {-invariant} \iff	i_1+m_1+n(i_2+m_2)\stackrel{\mathclap{\normalfont\mbox{\tiny{n+1}}}}{\equiv} 0  \tag {2.6}
		\end{align*}

\

\noindent Denote by $S(\CC^*\times \CC^*)^{\mathbb{Z}_{n+1}}$ the subalgebra of $\mathbb{Z}_{n+1}$-invariant differentials. 

\

From diagram (2.5) one obtains for each $r=0,...,n$:  

\begin{center}
		\begin{tikzpicture}
			\begin{scope}
				\node (tC) at (-0.7, 2.5) {$S(\CC^*\times \CC^*)^{\mathbb{Z}_{n+1}}$};
				\node (C) at (-4.7,0) {$S(\CC^*\times \CC^*)$};
				\node (X) at (-0.7, 0) {$S(X^*)$};\node (XXX) at (-8.3, 0) {};
				\node (XX) at (3.85, 0) {(2.7)};
				\draw[->, to path={-| (\tikztotarget)}] (X)  -- node[above] {$
				(\sigma \circ \phi_r)^*$}  node [below] {$\cong$} (C);
				\draw[->, to path={-| (\tikztotarget)}] (X) -- node[ right] {$\pi^*$}node[ left] {$\cong$} (tC);
				\draw[->,  to path={-| (\tikztotarget)}] (C) -- node[above] {$\varphi_r^*$}node[ below] {$\cong$} (tC);
			\end{scope} 
		\end{tikzpicture}
	\end{center}

\ 

The goal is to introduce a 3-gradation of $S(\CC^*\times \CC^*)$, with graded pieces spanned by symmetric differential monomials, that is respected by the isomorphisms $\varphi^*_r$. The structure of the pullback of symmetric differential monomials under the isomorphisms $\varphi_r^*:S(\CC^*\times \CC^*) \to S(\CC^*\times \CC^*)^{\mathbb{Z}_{n+1}}$ is the motivation for our choice of 3-gradations appearing below.

\

 The pullback by $\varphi_r$ of a monomial of degree $m$, using (2.2), is of the form:

		\begin{align*}
			\varphi_r^*(u_1^{i_{1}}u_2^{i_{2}}(du_1)^{m-q}(du_2)^q
				)= \sum_{l=0}^m c_{ql}(r)z_1^{j_1(i_{1},i_{2},m,q,r)+l}z_2^{j_2(i_{1},i_{2},m,q,r)-l}(dz_1)^{m-l}(dz_2)^l\tag{2.8}\label{eq:cql}
		\end{align*}

where 
			\begin{align*}
				j_1(i_{1},i_{2},m,q,r)&=(n+1-r)i_{1}+(r-n)i_{2}+(n-r)m+(2r-2n-1)q\\
				j_2(i_{1},i_{2},m,q,r)&=(-r)i_1+(r+1)i_2+(-r)m+(2r+1)q,
			\tag {2.9}\end{align*}
		
		 and the coefficients $c_{ql}(r)$ are determined by the equality:
		 \begin{align*}
		 	\left[(n+1-r)X-rY\right]^{m-q}\left [(r-n)X+(r+1)Y\right]^q = \sum_{l=0}^{m}c_{ql}(r)X^{m-l}Y^l.
		 \tag {2.10}\end{align*}

\ 

\noindent (all monomials on the right side of (2.8) are $\mathbb{Z}_{n+1}$-invariant).

\

The expression (2.8) jointly with (2.9) motivates the partition of the set of all monomials of degree $m$, say in the coordinates $(z_1,z_2)$, into the following $blocks$ of $m+1$ monomials:

\begin{align*}
    B_{k,i,m}:=\left\{
    z_1^{i-k+l}z_2^{k-l}(dz_1)^{m-l}(dz_2)^l)\right\}_{l=0,\dots,m},\tag {2.11}
\end{align*}

	\noindent where $i,k\in \mathbb{Z}$. The monomials in a single block are either   all $\mathbb{Z}_{n+1}$-invariant or all not:
	
	\begin{align*}
    B_{k,i,m}\cap S(\CC^*\times \CC^*)^{\mathbb{Z}_{n+1}} =
\begin{cases}
B_{k,i,m}, & \hspace {.05in}\text {if } 2k {\equiv} i+m \bmod{n+1}\\ 
\varnothing, & \hspace {.05in} \text {otherwise.}
 \end{cases}
\end{align*}
	
\

 When dealing with blocks of $\mathbb{Z}_{n+1}$-invariant monomials a distinct indexing will be useful. The $\mathbb{Z}_{n+1}$-invariance condition corresponds to $\hat k=\frac{2k-i-m}{n+1}\in \mathbb{Z}$. Set:

 \begin{align*}
			{\hat B}_{{\hat k},i,m}:=B_{ k({\hat k}),i,m}
   \hspace {.4in}
k(\hat k)=\frac{i+m}{2}+\frac{n+1}{2}\hat k \tag {2.12}
\end{align*}

\

\noindent The permissible indices $(\hat k,i,m)$ satisfy:

\begin{align*}
(n+1)\hat k\equiv i+m \bmod{2} \tag {2.13}
\end{align*}
\
Condition (2.13) has a dichotomy:

\begin{align*}
  \noindent  
(n+1)\hat k\equiv i+m \bmod{2}\iff \begin{cases}
 i+m\equiv 0 \bmod{2}& \hspace {.05in} n \text { odd}\\ 
 \hat k\equiv i+m \bmod{2}& \hspace {.05in} n \text { even}
 \end{cases}
\end{align*}

 \

 \subsubsection{3-gradations for $S(\CC^*\times \CC^*)$ and  $S(\CC^2)^{\mathbb{Z}_{n+1}}$}

 \

 \
 
 Set the 3-gradation of $S(\CC^*\times \CC^*)$ to be:

\begin {align*}
S(\CC^*\times \CC^*)=
\bigoplus_{\substack{ m\in \mathbb{Z}_{\ge 0}\\ i,k\in \mathbb{Z}}} V_{k,i,m}\tag {2.14}
\end{align*}

\noindent where the graded pieces are:

$$V_{k,i,m}:=\text{Span}(B_{k,i,m})$$  

\vspace {.2in}

\noindent ($V_{k,i,m}V_{k',i',m'}\subset V_{k+k',i+i',m+m'}$).

\ 

We have two variants of a 3-gradation of $S(\CC^*\times \CC^*)^{\mathbb{Z}_{n+1}}$:

\begin {align*}
S(\CC^*\times \CC^*)^{\mathbb{Z}_{n+1}}=\bigoplus_{\substack{ m\in \mathbb{Z}_{\ge 0}\\ i,k\in \mathbb{Z}\\ 2k- i-m\equiv 0 \bmod{n+1}}} V_{k,i,m}\tag {2.15}
\end{align*}

\noindent and the other using $\hat V_{\hat k,i,m}:=\text{Span}(\hat B_{\hat k,i,m})$ is:

\

\begin{align*}
 S(\CC^*\times \CC^*)^{\mathbb{Z}_{n+1}}=
 \bigoplus_{\substack{ m\in \mathbb{Z}_{\ge 0}\\ i,\hat k\in \mathbb{Z}\\(n+1)\hat k\equiv i+m \bmod{2}}} \hat V_{\hat k,i,m}
 \tag {2.16}
\end{align*}

\

For each $r=0,...,n$ the pullback morphisms $\varphi_r^*$ give the isomorphisms:

\begin{align*}S(\CC^*\times \CC^*) \xrightarrow[ ]{\varphi_r^*} S(\CC^*\times \CC^*)^{\mathbb{Z}_{n+1}}
\end{align*}  

\noindent that by (2.8) respects the 3-gradations. The relations (2.9) describe which graded pieces are sent to which graded pieces:

 \begin{align*}\varphi_r^*V_{k-r\left(\frac{2k-m-i}{n+1}\right), i+(n-2r)\left(\tiny{\frac{2k-m-i}{n+1}}\right), m}=V_{k, i, m}\end{align*}

 \ 

\noindent We will be interested in the following reformulation of the above:

\begin{align*}\varphi_r^*V_{\hat k, i, m, r}=\hat V_{\hat k, i, m}\end{align*}
\vspace{-.15in}
 \begin{align*}V_{\hat k, i, m, r}:=V_{\frac {i+m}{2}+(\frac{n+1}{2}-r)\hat{k}, i+(n-2r)\hat{k}, m}\tag{2.17}\label{eq:2.17}\end{align*}

 \
 
 The above gradation of $S(\CC^*\times \CC^*)$ induces a 3-gradation on $S(\CC^2)$ :
 
 \vspace {-.1in}
 
 \begin{align*}S(\CC^2)=\bigoplus_{\substack{ i,m\in \mathbb{Z}_{\ge 0}\\ 0\le k\le m+i}} V_{k,i,m}^{\text{reg}} 
				\tag {2.18}\end{align*}
 
 \vspace {-.05in}
 \noindent  The graded pieces are:
\vspace {-.05in}
 \begin{align*}V_{k,i,m}^{\text {reg}}:=V_{k,i,m}\cap H^0(\CC^2,S^m\Omega^1_{\CC^2})=\text{Span}(B_{k,i,m}^{\text{reg}})\end{align*}
	
\noindent where $B_{k,i,m}^{\text{reg}}:=\left\{z_1^{i_1}z_2^{i_2}dz_1^{m_1}dz_2^{m_2} \in B_{k,i,m}|i_1,i_2\ge 0\right\}.$ 

\

\noindent {\bf Remark:} The condition $0\le k\le m+i$ is equivalent to $B_{k,i,m}^{\text{reg}}\not =\emptyset$. 

\ 

We will also use with $j=1,2$:

$$B_{k,i,m}^{\text{reg},j}:=\left\{z_1^{i_1}z_2^{i_2}dz_1^{m_1}dz_2^{m_2} \in B_{k,i,m}\mid i_j\ge 0\right\}$$

\ 

For the algebra  $S(\CC^2)^{\mathbb{Z}_{n+1}}$ we will consider the 3-gradation:

\begin{align*}
 S(\CC^2)^{\mathbb{Z}_{n+1}}=
 \bigoplus_{\substack{ i,m\in \mathbb{Z}_{\ge 0}\\ |\hat k|\le \frac {i+m}{n+1}\\(n+1)\hat k\equiv i+m  \bmod{2}}} \hat V^\text{reg}_{\hat k,i,m}\tag {2.19}
\end{align*}

\noindent The condition $|\hat k|\le \frac {i+m}{n+1}$  is equivalent to $\hat B_{\hat k,i,m}^{\text{reg}}\not =\emptyset$.

\	

\

\subsection{Finding $\hslash^0(A_n,m)$ and its asymptotics}

\ 

\ 

 In this section we derive, for any fixed $n$, a formula for the function (notation of section 2.1):

\begin{align*}\hslash^0(A_n,m)=\dim \left( H^0(\tilde {X}\setminus E,S^m\Omega^1_{\tilde {X}})\middle/ H^0(\tilde {X},S^m\Omega^1_{\tilde {X}})\right ) 
\end{align*}

\noindent  giving the dimension of the space of obstructions for $m$-differentials on $\tilde X \setminus E$ to extend holomorphically through $E$. We will observe that $\hslash^0(A_n,m)$ is a quasi-polynomial of degree 3 in $m$. The main purpose is to determine the leading asymptotics:

\begin{align*}\hslash^0_\Omega(A_n)=\lim_{m\to \infty}\frac{\hslash^0(A_n,m)}{m^3}
\end{align*}

\

\subsubsection{From $\tilde X$ to $\CC^2$}

\ 

\

In this section, $\tilde {X}$ is the minimal resolution of the affine model $X$ of the $A_n$ singularity as described in section 2.1. We have the commutative diagram involving the resolution $\sigma$ and the smoothing $\pi$ of $X$:

\begin{equation*}
\begin{tikzcd}
 & (\CC^2, 0) \arrow[dl, "\varphi"', dashed] \arrow[d, "\pi"] \\
(\tilde{X}, E) \arrow[r, "\sigma"]  & (X, x)  \\
\end{tikzcd}
\end{equation*}

The map $\varphi$ induces the isomorphisms between the algebras of symmetric differentials:

\begin{align*} S\left(\tilde {X}\setminus E\right) \xrightarrow[\cong] {\varphi^*} S\left(\CC^2\right)^{\mathbb{Z}_{n+1}}   \text{ , } \text { } S\left(\tilde {X}\setminus \hat E\right) \xrightarrow[\cong] {\varphi^*} S\left(\CC^*\times \CC^*\right)^{\mathbb{Z}_{n+1}}
    \end{align*}

\noindent (recall: $\hat E=E+E_0+E_{n+1}$ as in 2.1). The first isomorphism follows from 
$$H^0\left(\tilde {X}\setminus E,S^m\Omega^1_{\tilde {X}}\right)\xleftarrow[\cong] {\sigma^*} H^0\left({X}\setminus \{x\},S^m\Omega^1_{ {X}}\right) \xrightarrow[\cong] {\pi^*} H^0\left(\CC^2\setminus \{0\},S^m\Omega_{\CC^2}^1\right)^{\mathbb{Z}_{n+1}}$$  
plus the equality 
$$H^0\left(\CC^2,S^m\Omega_{\CC^2}^1\right)^{\mathbb{Z}_{n+1}}=H^0\left(\CC^2\setminus \{0\},S^m\Omega_{\CC^2}^1\right)^{\mathbb{Z}_{n+1}}$$
due to the reflexivity of the sheaf $S^m\Omega_{\CC^2}^1$. The second isomorphism follows from similar arguments.

\

We make use of the following diagram:

\

\hspace {-.3in}\begin{tikzcd}
 H^0(\tilde {X}\setminus \hat E,S^m\Omega^1_{\tilde {X}})\arrow[rightarrow]{r}{\varphi^*}[swap]{\cong}
  & H^0(\CC^*\times \CC^*, S^m\Omega^1_{\CC^2})^{^{\mathbb{Z}_{n+1}}} \arrow[leftarrow]{r}{\varphi_r^*}[swap]{\cong}\arrow[leftarrow]{rd}[swap]{\varphi_r^*}&  H^0(\CC^*\times \CC^*,S^m\Omega^1_{\CC^2})
   \\
 H^0(\tilde {X}\setminus E,S^m\Omega^1_{\tilde {X}})\arrow[hookrightarrow]{u} \arrow[rightarrow]{r}{\varphi^*}[swap]{\cong} 
  & H^0(\CC^2, S^m\Omega^1_{\CC^2})^{^{\mathbb{Z}_{n+1}}} \arrow[hookrightarrow]{u} & H^0(\CC^2, S^m\Omega^1_{\CC^2}) \arrow[hookrightarrow]{u}
  \end{tikzcd}

  \

  \
  
  \noindent  to turn the task of finding $\hslash^0(A_n,m)$ from the spaces $\tilde X$ and $\tilde X\setminus E$ to $\CC^2$, via:

\begin{align*}\hslash^0(A_n,m)=\dim \left(H^0(\CC^2, S^m\Omega^1_{\CC^2})^{^{\mathbb{Z}_{n+1}}}\middle/ \bigcap_{r=0}^n \varphi_r^* H^0(\CC^2, S^m\Omega^1_{\CC^2})\right ) \tag {2.20}
\end{align*}

\noindent The above equality holds since a differential $w\in H^0(\tilde {X}\setminus E,S^m\Omega^1_{\tilde {X}})$ has no poles along $E_r$ and $E_{r+1}$ if and only if $\varphi^*w\in \varphi_r^* H^0(\CC^2, S^m\Omega^1_{\CC^2})$. 

\ 

Notice that whilst for each $r$, $\varphi_r^* H^0(\CC^2, S^m\Omega^1_{\CC^2}) \not \subset H^0(\CC^2, S^m\Omega^1_{\CC^2}) $ (differentials will be $\ZZ_{n+1}-invariant$ but not necessarily holomorphic along $\{z_1z_2=0\}$), one has 

$$\bigcap_{r=0}^n \varphi_r^* H^0(\CC^2, S^m\Omega^1_{\CC^2})\subset H^0(\CC^2, S^m\Omega^1_{\CC^2})$$

\noindent since $w\in \varphi_0^* H^0(\CC^2, S^m\Omega^1_{\CC^2})\cap \varphi_n^* H^0(\CC^2, S^m\Omega^1_{\CC^2})$ is $w=\varphi^*\tilde w$ with $\tilde w \in H^0(\tilde {X}\setminus E,S^m\Omega^1_{\tilde {X}})$.

\

\subsubsection{3-gradation breakdown of $\hslash^0(A_n,m)$ }

\ 

\

The 3-gradation breakdown of $\hslash^0(A_n,m)$ follows from the the pullback mappings, $\varphi_r^*$, sending graded piece to graded piece of the 3-gradation of $H^0(\CC^*\times \CC^*, S^m\Omega^1_{\CC^2})$ given in (2.14).

\  

Let $w\in  H^0(\CC^2, S^m\Omega^1_{\CC^2})^{^{\mathbb{Z}_{n+1}}}$. Since no monomial belongs to two blocks that generate two distinct graded pieces, the following are equivalent: 

\begin{enumerate}
    \item [i)] $(\varphi^*)^{-1} w$ extends regularly along  $E$ 
    \item [ii)] $(\varphi^*)^{-1} w_{\hat k,i,m}$ extends regularly along  $E$, $\forall$ $w_{\hat k,i,m}$ in the $(\hat k,i,m)$-decomposition of $w$.
\end{enumerate}

Also relevant to this equivalence and what will follow is the decomposition  for each $r=0,...,n$:

$$\text { }\text { }\text { }\text { }\text { }\text { }\text { }\text { }\text { }\text { }\text { }(\varphi_r^*)^{-1}w_{\hat k,i,m}=\sum w_{\hat k,i,m,r} \text { },  \text { }\text { } \text { }\text { }\text { } w_{\hat k,i,m,r}\in V_{\hat k,i,m,r}$$

 \ 

 \noindent with this decomposition, the following statements are equivalent to $(\varphi^*)^{-1}w_{\hat k,i,m}$ extending regularly along  $E$:

\begin{enumerate}
     
    \item [a)] $w_{\hat k,i,m,r}\in V^\text {reg}_{\hat k,i,m,r}$, for $r=0,...,n$
    \item [b)] $w_{\hat k,i,m,r}\in V^\text {reg,1}_{\hat k,i,m,r}$, for $r=-1,...,n$
    \item [c)] $w_{\hat k,i,m}\in \bigcap_{r=-1}^n\varphi_r^*V_{\hat k,i,m,r}^{\operatorname{reg},1}$
\end{enumerate}

\noindent See the note below (2.4) for the meaning of the case $r=-1$. The artificial case of $r=-1$ in (b) and (c) will be used to streamline formulas. We observe that $\varphi_{r-1}^*V^\text {reg,1}_{\hat k,i,m,r-1}=\varphi^*_rV^\text {reg,2}_{\hat k,i,m,r}$, for $r=0,...,n+1$.

\

\begin{proposition} Let $(\tilde X,E)$ be the germ of the minimal resolution of $(X,x)$  the germ of an $A_n$ singularity. Then:

\begin{align*}\hslash^0(A_n,m)= \sum_{\substack{ 0\le i\le mn-1\\ |\hat k|\le \frac {i+m}{n+1}\\(n+1)\hat k\equiv i+m  \bmod{2}}} \hslash^0(A_n,\hat k,i,m)\tag {2.21}\end{align*}

\noindent where $\hslash^0(A_n,\hat k,i,m):=\dim \left(\frac{\hat V_{\hat k,i,m}^{\text{reg}}}{\bigcap_{r=-1}^n\varphi_r^*V_{\hat k,i,m,r}^{\operatorname{reg},1}}\right)$.

\end{proposition}

   \ 

   \

\noindent \textbf{Note}. The inclusion $\bigcap_{r=-1}^n\varphi_r^*V_{\hat k,i,m,r}^{\operatorname{reg},1}\subset \hat V_{\hat k,i,m}^{\text{reg}}$ holds since $\hat V_{\hat k,i,m}^{\text {reg}}=\varphi_{-1}^*V_{\hat k,i,m,-1}^{\operatorname{reg},1}\cap \varphi_n^*V_{\hat k,i,m,n}^{\operatorname{reg},1}$.

   \ 

  \begin{proof} The result follows from (2.20), the decomposition:
  
  $$H^0(\CC^2, S^m\Omega^1_{\CC^2})^{^{\mathbb{Z}_{n+1}}}=\bigoplus_{\substack{ i\in \mathbb{Z}_{\ge 0}\\ |\hat k|\le \frac {i+m}{n+1}\\(n+1)\hat k\equiv i+m  \bmod{2}}} \hat V^\text{reg}_{\hat k,i,m}$$
  
  \noindent and from the equivalences i) $\iff$ ii) and a) $\iff$  b) $\iff$ c) that imply:

  $$\bigcap_{r=0}^n \varphi_r^* H^0(\CC^2, S^m\Omega^1_{\CC^2})=\bigoplus_{\substack{ i\in \mathbb{Z}_{\ge 0}\\ |\hat k|\le \frac {i+m}{n+1}\\(n+1)\hat k\equiv i+m  \bmod{2}}} \bigcap_{r=-1}^n\varphi_r^*V_{\hat k,i,m,r}^{\operatorname{reg},1}$$

  For the upper bound $i\le mn-1$, see \hyperref[thm:3]{Theorem 3}(b).
 \end{proof}

\

\

\subsubsection{Independence of the conditions to have no poles along the $E_r$ and the formula for $h^0(A_n,\hat k,i,m)$}

\ 

\

It follows from Proposition 1  that  we need to get a handle on the subspaces $\varphi_r^*V_{\hat k,i,m,r}^{\text{reg},1} \subset \hat V_{\hat k,i,m}$, $r=-1,...,n$  and determine how they intersect each other.

\ 

We will show that the subspaces $\varphi_r^*V_{\hat 
 k,i,m,r}^{\text{reg},1}$ of $\hat V_{\hat k,i,m}$ are in general position, which implies that the dimension of their intersection  is the expected dimension. It also means that the conditions for a given differential not to have poles along the components $E_r$ are independent of each other.

\

\begin{lemma} [Independence of the conditions to have no poles along the $E_r$] The subspaces $\varphi_r^*V_{\hat k,i,m,r}^{\operatorname{reg,1}}\subset \hat V_{\hat k,i,m}$, $r=-1,...,n$, of $\hat V_{\hat k,i,m}$ are in general position.
\end{lemma}

\begin{proof} Set $\mu_q:=z_1^{i_1(\hat k,i,m)+q}z_2^{i_2(\hat k,i,m)-q}(dz_1)^{m-q}(dz_2)^q$, $q=0,...,m$, the monomials spanning $\hat {V}_{\hat k,i,m}$ and $\mu_{q,r}:=u_1^{i_1(\hat k,i,m,r)+q}u_2^{i_2(\hat k,i,m,r)-q}(du_1)^{m-q}(du_2)^q$, $q=0,...,m$, the monomials spanning ${V}_{\hat k,i,m,r}$.

\ 

Consider the commutative diagram of isomorphisms:

\begin{equation*}
\begin{tikzcd}
 V_{\hat k,i,m,r}\arrow[r, "\varphi_r^*"]\arrow[dr, "\Psi_r"']& \hat {V}_{\hat k,i,m} \arrow[d, "\Psi"]  \\
  & \CC[X,Y]_{(m)}  \\
\end{tikzcd}
\end{equation*}

\noindent where the $\CC$-linear map $\Psi$ is defined by $\Psi(\mu_{q})=X^{m-q}Y^q$.

\ 

From (2.8) and (2.10) of section 2.2, it follows that:

$$\Psi_r(\mu_{r,q})=\left[(n+1-r)X-rY\right]^{m-q}\left [(r-n)X+(r+1)Y\right]^q $$

\noindent and hence the zero locus  $Z\left(\Psi_r(\mu_{r,q})\right)=(m-q)\text { p}_{r,1}+q\text { p}_{r,2}\subset \PP^1$, where $\text{p}_{r,1}=[r:n+1-r]$ and $\text{p}_{r,2}=[r+1:r-n]$.

\ 

The monomials $\mu_{r,q}$ with no poles along $E_{r+1}\cap  U_r=\{u_1=0\}$, generating $V^{reg,1}_{\hat k,i,m,r}$, have $q\ge\max \{0,-i_1(\hat k,i,m,r)\}$, hence:

\begin{align*}
    \Psi_r(V^{reg,1}_{\hat k,i,m,r})= \big \{ P\in \CC[X,Y]\text { }\big | \text { } Z(P)\ge \max \{0,-i_1(\hat k,i,m,r)\} \text { p}_{r,2}\big \}
\end{align*}

Since the points $\text{p}_{r,2}$, $r=-1,...,n$, are all distinct, the membership conditions for $P\in \Psi_r(V^{reg,1}_{\hat k,i,m,r})$, $r=-1,...,n$, are independent. Therefore the subspaces $\Psi_r(V^{reg,1}_{\hat k,i,m,r})\subset \CC[X,Y]_{(m)}$ are in general position and the result follows.

\end{proof}

\ 

\

\begin{corollary}  For each triple $(\hat k,i,m)$, the contribution $\hslash^0(A_n,\hat 
k,i,m)$  to $\hslash^0(A_n,m)$ is:
\vspace {-.1in}
\begin{align*}\hslash^0(A_n,\hat k,i,m)= \min \left\{\sum_{r=0}^{n-1}\operatorname{codim } (V_{\hat k,i,m,r}^{_{\operatorname{reg},1}},V_{\hat k,i,m,r})\text { }, \text { }\dim \hat V_{\hat k,i,m}^{_{\operatorname {reg}}}\right\}
\tag {2.22}\end{align*}

\noindent where:
\begin{itemize}
    \item [i)]  $\operatorname{codim }  (V_{\hat k,i,m,r}^{_{\operatorname{reg},1}},V_{\hat k,i,m,r})=\max \left\{0,\frac{m-i}{2}+\frac{2r-n+1}{2}\hat k\right\}\le m $    \hspace {.42in} $\operatorname{(2.23)}$

            \item [ii)]   $\dim \hat V_{\hat k,i,m}^{_{\text {reg}}}=\max \left\{0,m+1-\sum_{r=-1,n}\operatorname{codim } (V_{\hat k,i,m,r}^{_{\operatorname{reg},1}},V_{\hat k,i,m,r})\right\}$
\end{itemize}

\end{corollary}

\begin{proof} The formula (2.22)  follows directly from the subspaces $\varphi_r^*V_{\hat k,i,m,r}^{_{\operatorname{reg},1}}$ of $\hat V_{\hat k,i,m}$ being in general position (lemma 1) plus the fact that $\hslash^0(A_n,\hat k,i,m)=\operatorname{codim }  (\bigcap_{r=-1}^n\varphi_r^*V_{\hat k,i,m,r}^{\operatorname{reg},1} \text { }, \text { }\hat V_{\hat k,i,m}^{_{\operatorname {reg}}})$.

\

    Towards the expression (2.30), we have that the monomials spanning $V_{\hat k,i,m,r}$ are $u_1^{i_1+l}u_2^{i_2-l}dz_1^{m-l}dz_2^{l}$ where $l=0,...,m$ and  $i_1=\frac {i-m}{2}+\frac{-2r+n-1}{2}\hat k$ (a consequence of (\ref{eq:2.17}). Hence there are $\max \{0,-i_1\}$ monomials which are not regular along $\{u_1=0\}$. Note that in the range $|\hat k|\le \floor{\frac {i+m}{n+1}}$ we have $\max \{0,-i_1\}\le m$.

    \ 

    Identity ii) follows from the note after proposition 1 and lemma 1.
\end{proof}

\

\

\subsubsection{The geometry of $h^0(A_n,\hat k,i,m)$}

\ 

\

In the previous section,  the function $\hslash^0(A_n,\hat k,i,m)$  was fully determined via (2.22) and (2.23) of corollary 1. To understand $\hslash^0(A_n,m)$ the geometric properties of the function $\hslash^0(A_n,\hat k,i,m)$, when $n$ and $m$ are fixed, play a important role and this section uncovers them.

 \

The function $\hslash^0(A_n,\hat k,i,m)$ is an even function relative to $\hat k$, a consequence of $\operatorname{codim }  (V_{\hat k,i,m,r}^{_{\operatorname{reg},1}},V_{\hat k,i,m,r})=\operatorname{codim }  (V_{-\hat k,i,m,n-1-r}^{_{\operatorname{reg},1}},V_{-\hat k,i,m,n-1-r})$. 

\ 

Fix $n$, using corollary 1 we associate to each $m\ge 0$  a polygonal region $\mathcal{P}_n(m)$ in the $(i,\hat{k})-plane$. The polygon $\mathcal{P}_n(m)$ is the closure of the region where $\hslash^0(A_n,\hat k,i,m)$ is nonvanishing (with $i$ and $\hat{k}$ being considered as continuous parameters). The polygon $\mathcal{P}_n(m)$ has the polygonal decomposition 

$$\mathcal{P}_n(m)=\cup_{j=1}^{2n+2}\mathcal{P}_n^j(m)$$

\noindent where the $\mathcal{P}_n^j(m)$, $j=1,...,2n+2,$ are the regions  where $\hslash^0(A_n,\hat k,i,m)$ is given by a single affine function on $i$ and $\hat{k}$, making this decomposition unique, see below for more details.

\ 

\begin{remark} All the polygons $\mathcal{P}_n^j(m)$ are convex with the exception of two. One can do a convex  subdivision of $\mathcal{P}_n(m)=\cup_{j=1}^{2n+2}\mathcal{P}_n^j(m)$ in a natural way,

$$\mathcal{P}_n(m)=\cup_{r=1}^{4n}\mathcal{P'}_n^r(m)$$

 \noindent and more importantly the convex polygons $\mathcal{P}_n^{'r}(m)$ change with $m$ via $\mathcal{P}_n^{'r}(m)=\{x\in \RR^2 \text { } \big | \text { }   \mu_{r,l}(x)\le a_{r,l}m+b_{r,l},l=1,...,k_r, a_{r,l},b_{r,l}\in \QQ\}$ with $\mu_{r,l}$  linear forms  with $\ZZ$-coefficients. This type of decomposition will allow the use of the theory of polynomial weighted lattice sums over convex polytopes  (see \cite{Ehrhart}, \cite{Brion}, \cite{Vergne}) to derive properties of the function $\hslash^0(A_n,m)$, such as being a quasi-polynomial of degree 3 in $m$.

\end{remark}

\ 

The fact, that there are $2n+2$  polygons $\mathcal{P}_n^j(m)$ follows from the discussion below and the properties of the function $\hslash^0(A_n,\hat k,i,m)$ derived from corollary 1. The polygons $\mathcal{P}_n^j(m)$ have some of the properties described in the remark:   the coordinates of the vertices are affine functions of $m$ and the slopes of the edges are independent of $m$.

\ 

 To determine the polygonal decomposition defined above, we need to consider:
 
  \ 
 
 1) The $n+2$ lines coming from the components $E_r$, $r=0,...,n+1$, of $\hat E$. These lines are given by $\frac{m-i}{2}+\frac{2(r-1)-n+1}{2}\hat k=0$  and separate the half planes where the $\operatorname{codim }  (V_{\hat k,i,m,r-1}^{_{\operatorname{reg},1}},V_{\hat k,i,m,r-1})$ equals either $\frac{m-i}{2}+\frac{2(r-1)-n+1}{2}\hat k$ or $0$. All the lines pass through the point $(m,0)$ in the $(i,\hat k)$-plane. The relevant half plane $\{i\ge 0\}$ is therefore separated in $2(n+2)$ radial sectors with center the point  $(m,0)$.

 \ 
 
 2)  Two  extra lines are required to determine the regions where $\dim \hat V_{\hat k,i,m}^{_{\text {reg}}}$ either $= 0$ or $=m+1-\sum_{r=-1,n}\operatorname{codim } (V_{\hat k,i,m,r}^{_{\operatorname{reg},1}},V_{\hat k,i,m,r})$). The two lines of 1) associated with $E_0$ and $E_{n+1}$ are also required to determine these regions. We have the following:

    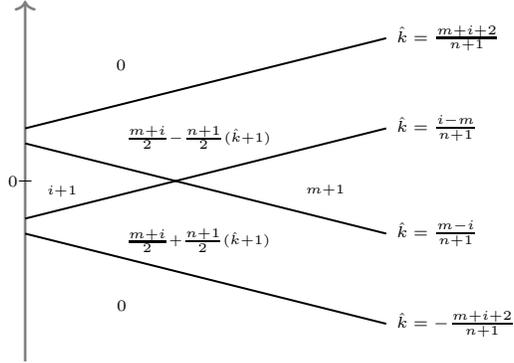
\begin{figure}[h!]
    \centering
		\begin{tikzpicture}[scale=0.4,domain=0:12]
			
			\draw[->,color=gray, line width=1] (0,-6) -- (0,6);
			
			\draw[color=black, line width = 0.75]    plot (\x,1*\x/4-5/4) node[right]{\tiny ${\hat{k}= \frac{i-m}{n+1}}$};
			\draw[color=black, line width = 0.75]    plot (\x,-1*\x/4-7/4) node[right]{\tiny ${\hat{k}= -\frac{m+i+2}{n+1}}$};
			\draw[color=black, line width = 0.75]    plot (\x,-1*\x/4+5/4) node[right]{\tiny ${\hat{k}= \frac{m-i}{n+1}}$};
			\draw[color=black, line width = 0.75]    plot (\x,1*\x/4+7/4) node[right] {\tiny ${\hat{k}= \frac{m+i+2}{n+1}}$};

			\draw (1.25,0.2) node[anchor=north] {\tiny $_{i+1}$};
			\draw (3,2.167316223304202) node[anchor=north west] {\tiny $\displaystyle _{\frac{m+i}2 - \frac{n+1}2 (\hat{k}+1)}$};
			\draw (3,-1.2523974035576506) node[anchor=north west] {\tiny $\displaystyle _{\frac{m+i}2 + \frac{n+1}2 (\hat{k}+1)}$};
			\draw (10,0.2) node[anchor=north] {\tiny $_{m+1}$} ;
			\draw (2.6881882117080997,4.3603934405308244) node[anchor=north west] {\tiny $0$};
			\draw (2.7124163993584607,-3.6437188890081487) node[anchor=north west] {\tiny $0$};
			
			\draw     (-0.21,0) -- (0.21,0) node[left] {\tiny $0\,$} ;
		\end{tikzpicture}
				\caption{Graph of $\dim \hat V_{\hat k,i,m}^{_{\text {reg}}}$ on the $(\hat{k},i)$ plane}
    \end{figure}
  
  \
  
  3) The $2(n+2)$ sectors defined in 1), except for three between the lines associated to $E_{n}$,$E_{n+1}$,$E_{0}$ and $E_1$ and where $\sum_{r=0}^{n-1}\operatorname{codim } (V_{\hat k,i,m,r}^{_{\operatorname{reg},1}},V_{\hat k,i,m,r})=0$, each have a line segment separating the regions where  $\hslash^0(A_n,\hat k,i,m)=\sum_{r=0}^{n-1}\operatorname{codim } (V_{\hat k,i,m,r}^{_{\operatorname{reg},1}},V_{\hat k,i,m,r})$ or $\hslash^0(A_n,\hat k,i,m)=\dim \hat V_{\hat k,i,m}^{_{\operatorname {reg}}}$.

  \ 

  4) Altogether there are $2n+2$ polygons $\mathcal{P}_n^j(m)$ where $\hslash^0(A_n,\hat k,i,m)$  as  a function  in $i$ and $\hat k$ is defined by a single  nontrivial affine expression (the case $n=3$ is illustrated in the next figure). This follows from 1), 2) and 3) plus the fact that $\dim \hat V_{\hat k,i,m}^{_{\text {reg}}}$ has the same expression in the sectors bonded by the lines $E_{n}$, $E_{n+1}$, $E_{0}$ and $E_1$. Recall, that there is the symmetry coming from $\hslash^0(A_n,\hat k,i,m)$ being an even function relative to $\hat k$.

  \

  5) The above are statements are for fixed $n$ and independent of $m$. As for dependence in $m$. The coordinates of the vertices of the $2n+2$ polygonal regions are affine functions in $m$ with rational coefficients. Also important if one wants to give a complete description of the function $\hslash^0(A_n,m)$, the slopes of the boundary lines are independent of $m$.

  \ 

  6) The affine functions $\hslash^0(A_n,\hat k,i,m)_j:=\hslash^0(A_n,\hat k,i,m)|_{\mathcal{P}_n^j(m)}$ have the following structure for each $j=1,...,2n+2$, 

  \begin{align*}
      \hslash^0(A_n,\hat k,i,m)_j=a_j(n)i+b_j(n)\hat{k}+c_j(n)m+d_j(n) \tag {2.24}
  \end{align*}

  \noindent $a_j(n),b_j(n),c_j(n),d_j(n)\in \QQ$. It is relevant to note the fact that the coefficients in $i$ and $\hat{k}$ are independent of $m$.

\newpage
 
Below we illustrate the function $\hslash^0(A_n,\hat k,i,m)$ as a function of $\hat k$ and $i$ for $n=3$.  All key features of the general case are present. 

\begin{figure}[h]
			\tikzset{every picture/.style={line width=0.75pt}} 
			
			\begin{tikzpicture}[x=0.75pt,y=0.75pt,yscale=-1,xscale=1,color=gray]
				
				\draw [color=black]  [line width=1.25]    (33.5,301) -- (33.5,184.3) ;
				
				\draw [color=black]  [line width=0.75]    (32.9,184.26) -- (466.33,9.98) ;
				\draw  [color=black]   (32.9,188.81) -- (118.94,217.95) ;
	\draw  [color=gray, line width=0.25, dashed]   (118.94,217.95) -- (466.33,359.948) ;
				\draw [color=black]   (130.3,203.67) -- (466.33,474.74) ;
		\draw [color=gray, line width=0.25, dashed]   (32.9,125.099) -- (466.33,474.74) ;
				\draw  [color=black]   (32.9,295.45) -- (118.77,266.77) ;
	\draw  [color=gray, line width=0.25, dashed]   (118.77,266.77) -- (466.33,123.719) ;
				\draw  [color=black]   (129.56,281.04) -- (466.33,9.98) ;
		\draw  [color=gray, line width=0.25, dashed]   (129.56,281.04) -- (32.9,358.839) ;
				
				\draw [color=black]  [line width=0.75]    (32.9,300.45) -- (466.33,474.74) ;
				\draw [color=black  ,draw opacity=1 ]   (118.94,217.95) -- (118.77,242.36) ;
				\draw [color=black]    (118.94,217.95) -- (130.3,203.67) ;
				\draw [color=black]  (130.3,203.67) -- (177.24,163.1) ;
				\draw [color=black]    (177.24,163.1) -- (466.33,9.98) ;
				\draw [color=black]   (118.94,242.36) -- (118.77,266.77) ;
				\draw [color=black]   (118.77,266.77) -- (129.56,281.04) ;
				\draw [color=black]   (177.99,322.39) -- (129.56,281.04) ;
				\draw [color=black]    (177.99,322.39) -- (466.33,474.74) ;
				
				\draw     (38.07,9.98) -- (27.99,9.98) ;
				\draw    (32.9,1.32) -- (32.9,472.15) ;
				\draw  (37.87,466.35) -- (27.78,466.35) ;
				\draw  [color=black]   (178.13,163.32) -- (178.13,202.22) -- (177.99,247.32) ;
		\draw  [color=gray, dashed, line width=0.5]   (178.13,163.32) -- (178.13,9.8) ;
	
				\draw  [color=black]   (177.99,259.99) -- (177.99,322.39) ;
		\draw  [color=gray, dashed, line width=0.5]   (178.05,322.39) -- (178.05,474) ;
				\draw    (37.57,184.09) -- (27.49,184.09) ;
				\draw [color=gray]    (467.46,237.11) -- (467.46,247.4) ;
				\draw    (38.07,300.42) -- (27.99,300.42) ;
				
				\draw (-3,4.5) node [anchor=north west][inner sep=0.75pt]  [font=\normalsize,xscale=0.75,yscale=0.75]  {\color{gray}$m+1$};
				\draw (-4,171.53) node [anchor=north west][inner sep=0.75pt]  [font=\normalsize,xscale=0.75,yscale=0.75]  {\color{gray}$\displaystyle\frac{m+2}{4}$};
				\draw (448.12,246.61) node [anchor=north west][inner sep=0.75pt]  [font=\normalsize,xscale=0.75,yscale=0.75]  {\color{gray}$3m+2$};
				\draw (-12,461.87) node [anchor=north west][inner sep=0.75pt]  [font=\normalsize,xscale=0.75,yscale=0.75]  {\color{gray}$-m-1$};
				\draw (170.87,249.61) node [anchor=north west][inner sep=0.75pt]  [font=\normalsize,xscale=0.75,yscale=0.75]  {\color{gray}$m$};
				\draw (-12,287.83) node [anchor=north west][inner sep=0.75pt]  [font=\normalsize,xscale=0.75,yscale=0.75]  {\color{gray}$-\displaystyle\frac{m+2}{4}$};

				\draw (63.33,170.07) node [anchor=north west][inner sep=0.75pt]  [xscale=0.75,yscale=0.75]  {\color{black}$\displaystyle\frac{m+i}{2} -2\hat{k} +1$};
				\draw (54,236.61) node [anchor=north west][inner sep=0.75pt]  [xscale=0.75,yscale=0.75]  {\color{black}$\displaystyle i+1$};
				\draw (63.33,288.73) node [anchor=north west][inner sep=0.75pt]  [xscale=0.75,yscale=0.75]  {\color{black}$\displaystyle\frac{m+i}{2} +2\hat{k} +1$};
				\draw (120.66,230.55) node [anchor=north west][inner sep=0.75pt]  [xscale=0.75,yscale=0.75]  {\color{black}$\displaystyle\frac{3}{2}( m-i)$};
				\draw (184.73,169.93) node [anchor=north west][inner sep=0.75pt]  [xscale=0.75,yscale=0.75]  {\color{black}$\displaystyle\frac{m-i}{2} +\hat{k}$};
				\draw (184.73,287.93) node [anchor=north west][inner sep=0.75pt]  [xscale=0.75,yscale=0.75]  {\color{black}$\displaystyle\frac{m-i}{2} -\hat{k}$};
				\draw (142.72,273.48) node [anchor=north west][inner sep=0.75pt]  [font=\footnotesize,rotate=-40.21,xscale=0.75,yscale=0.75]  {\color{black}$m-i-\hat{k}$};
				\draw (146.82,194.99) node [anchor=north west][inner sep=0.75pt]  [font=\footnotesize,rotate=-40.52,xscale=0.75,yscale=0.75]  {\color{black}$m-i+\hat{k}$};
    \node[] at (460.5,140) {$E_4$};
    \node[] at (460.5,345) {$E_0$};
    \node[] at (420.5,420) {$E_1$};
    \node[] at (420.5,65) {$E_3$};
    \node[] at (190.5,30) {$E_2$};
				
\end{tikzpicture}
		\caption{The function  $\hslash^0(A_3,\hat k,i,m)$ with $m$ fixed on the $(i,\hat{k})$-plane. Recall that $\hslash^0(A_3,\hat k,i,m)$ is only defined at integral points with $4\hat k\equiv i+m \bmod{2}$.}
		
	\end{figure}

\newpage

\subsubsection{Closed formula for the asymptotics of $\hslash^0(A_n,m)$}

\ 

\

As mentioned in the introduction, one of the aims of this paper is to find the contribution that each $A_n$ singularity gives towards the the $m$-growth asymptotics of $h^1(X,S^m\Omega_X)$ when $X$ is the minimal resolution of a surface of general type $Y$ with canonical singularities.  We saw in section 1 that this contribution can be derived from $\hslash^0_\Omega(A_n)=\lim_{m\to \infty}\frac{\hslash^0(A_n,m)}{m^3}$. In this section we give a closed formula for $\hslash^0_\Omega(A_n)$.

\

\begin{proof} (of \hyperref[thm:1]{Theorem 1}) Part (a) follows directly from  Proposition 1

\begin{align*}\hslash^0(A_n,m)= \sum_{\substack{ 0\le i\le mn-1\\ |\hat k|\le \frac {i+m}{n+1}\\(n+1)\hat k\equiv i+m  \bmod{2}}} \hslash^0(A_n,\hat k,i,m), \tag {2.27} \end{align*}

\noindent corollary 1 describing $\hslash^0(A_n,x_2,x_1,m)(=h_{n,m}(x_1,x_2))$ and  the polygonal geometry associated to $\hslash^0(A_n,\hat k,i,m)$ described in 2.4.4.

\ 

Part b) asserts that:
   
 \begin{align*} \hslash^0_\Omega(A_n):=\lim_{m\to \infty}\frac{\hslash^0(A_n,m)}{m^3}=\frac {4}{3}\sum^n_{j=1}\frac {1}{j^2}-\frac{12n^4+65n^3+117n^2+72n}{6(n+1)^2(n+2)^2}.\tag {2.25}\end{align*}

\noindent and
\begin{align*} \hslash^0(A_n,m)=\hslash^0_\Omega(A_n)m^3+3\hslash^0_\Omega(A_n)m^2+O(m),\tag {2.26}\end{align*}

\ 

Consider:

$$I(A_n,m):=\iint_{\mathcal{P}_n(m)}h_{n,m}(x)dx$$

\ 

\

Claim 1:  $$\sum_{ x\in \mathcal{P}_n(m)\cap \ZZ^2} h_{n,m}(x)=I(A_n,m)+O(m)$$

\ 

\ 

If $x=(x_1,x_2)$, set $S_x:=[x_1-\frac{1}{2},x_1+\frac{1}{2}]\times [x_2-\frac{1}{2},x_2+\frac{1}{2}] $. Consider:

$$C_{\mathcal{P}_n(m)}=\{x\in \mathcal{P}_n(m)\cap \ZZ^2\text { | } S_x\subset \mathcal{P}_n(m)\}$$

$$PC_{\mathcal{P}_n(m)}=\{x\in \mathcal{P}_n(m)\cap \ZZ^2\text { | } S_x\not\subset \mathcal{P}_n(m)\}$$

$$O_{\mathcal{P}_n(m)}=\{x\in \ZZ^2\setminus (\mathcal{P}_n(m)\cap \ZZ^2)\text { | } S_x\cap \mathcal{P}_n(m)\neq \emptyset\}$$

\

Note $\mathcal{P}_n(m)\cap \ZZ^2=C_{\mathcal{P}_n(m)}\cup PC_{\mathcal{P}_n(m)}$ and:

$$\mathcal{P}_n(m)\subset\cup_{x\in C_{\mathcal{P}_n(m)}\cup PC_{\mathcal{P}_n(m)}\cup O_{\mathcal{P}_n(m)}}S_x $$

\ 

We start by showing:

\begin{align*}
    \sum_{ x\in \mathcal{P}_n(m)\cap \ZZ^2} h_{n,m}(x)=\sum_{x\in \mathcal{P}_n(m)\cap \ZZ^2}\iint_{S_x}h_{n,m}(y)dy + O(m) \tag {2.28}
\end{align*}

\noindent Note: $h_{n,m}(y)=0$ if $y\not \in \mathcal{P}_n(m)$. 

\

We need a further decomposition of $C_{\mathcal{P}_n(m)}$ with:

$$SC_{\mathcal{P}_n(m)}=\{x\in C_{\mathcal{P}_n(m)} { | } \exists j, S_x\subset \mathcal{P}_n^j(m)\}$$

\noindent and $MC_{\mathcal{P}_n(m)}=C_{\mathcal{P}_n(m)}\setminus SC_{\mathcal{P}_n(m)}$. The points $x\in C_{\mathcal{P}_n(m)}$,  $S_x$ contained in a single polygon $\mathcal{P}_n^j(m)$ and hence $h_{n,m}(x)$  given by a single affine function on $S_x$, satisfy:

\begin{align*}
   h_{n,m}(x)=\iint_{S_x} h_{n,m}(y) , \hspace {.2in} \text { } \forall x\in SC_{\mathcal{P}_n(m)} \tag {2.29}
\end{align*}

If $x\in MC_{\mathcal{P}_n(m)}$ the equality above doesn't hold, but:

\begin{align*}
    \big |h_{n,m}(x)-\iint_{S_x} h_{n,m}(y)\big |< C(n) , \hspace {.2in} \text { } \forall x\in MC_{\mathcal{P}_n(m)} \tag {2.30}
\end{align*}

\noindent where $C(n)=\max \{|a_j(n)|,|b_j(n)|;j=1,...,2n+2\}$, with $a_j(n)$ and $b_j(n)$ from (2.24), is a bound on the partial derivatives of $\hslash^0(A_n,y,m)$
(bound independent from $m$).

\ 

Let $x\in PC_{\mathcal{P}_n(m)}$, using the fact that  $0\le h_{n,m}(z)\le 1$ $\forall z\in \partial \mathcal{P}_n(m)$ and  $C(n)$, it follows that $\forall y \in S_x$, $|h_{n,m}(y)|<1+C(n)$  ($h_{n,m}(xy=0$ if $y\not \in \mathcal{P}_n(m)$). Hence:

\begin{align*}
    \big |h_{n,m}(x)-\iint_{S_x} h_{n,m}(y)\big |< 1+ C(n) , \hspace {.2in} \text { } \forall x\in PC_{\mathcal{P}_n(m)} \tag {2.31}
\end{align*}

The subclaim (2.28) follows from $\# (MC_{\mathcal{P}_n(m)}\cup PC_{\mathcal{P}_n(m)})=O(m)$, it grows with $m$  as the sum of the lengths of all edges of the polygonal decomposition of $\mathcal{P}_n(m)$.

\ 

To prove the claim 1, it remains to show that:

$$\sum_{x\in  O_{\mathcal{P}_n(m)}}\iint_{S_x}h_{n,m}(y)dy=\iint_{\mathcal{P}_n(m)}h_{n,m}(y)dy-\sum_{x\in \mathcal{P}_n(m)\cap \ZZ^2}\iint_{S_x}h_{n,m}(y)dy=O(m)$$

\noindent which follows from $|h_{n,m}(y)|<1+C(n)$ $\forall y\in S_x$ if $x\in O_{\mathcal{P}_n(m)}$ and $\#  O_{\mathcal{P}_n(m)}=O(m)$.

\

Claim 2:  $$\sum_{ x\in \mathcal{P}_n(m)\cap \ZZ^2} h_{n,m}(x)=2 \hslash^0(A_n,m) + O(m)$$

\ 

\ 

Set $$\hslash^0(A_n,\hat {k},m):=\sum^{n(m+1)-1}_{\substack{i=0\\(n+1) \hat{k}\equiv i+m  \bmod{2}}} \hslash^0(A_n,\hat{k},i,m)$$

\noindent then $\hslash^0(A_n,m)=\sum_{|\hat{k}|\le m+1}\hslash^0(A_n,\hat {k},m)$. It is enough to show $\exists$ $A(n)>0$ such that:

\begin{align*}
    \left|\sum^{n(m+1)-1}_{i=0} \hslash^0(A_n,\hat{k},i,m)-2\hslash^0(A_n,\hat {k},m)\right|<A(n) \hspace{.3in} \forall |\hat{k}|\le m+1 \tag {2.32}
\end{align*}

For fixed $\hat{k}$ (and also $n$ and $m$), the function $\hslash^0(A_n,\hat{k},i,m)$ is a piecewise  affine function in $i$ with at most $n+2$ affine pieces (and continuous if $i$ is considered a continuous parameter).

\ 

One has $\hslash^0(A_n,\hat{k},i,m)=\frac{1}{2}[\hslash^0(A_n,\hat{k},i-1,m)+\hslash^0(A_n,\hat{k},i+1,m)]$ if $i-1,i,i+1$ is the same affine piece. Otherwise one has  $|\hslash^0(A_n,\hat{k},i,m)-\frac{1}{2}[\hslash^0(A_n,\hat{k},i-1,m)+\hslash^0(A_n,\hat{k},i+1,m)]|\le C(n)$ with $C(n)$ as above. One also has the boundary condition that $0\le \hslash^0(A_n,\hat{k},0,m),\hslash^0(A_n,\hat{k},n(m+1)-1,m)\le 1$. All the above implies $A(n)=2+(n+2)C(n)$ works in (2.32), with $2+C(n)$ bounding the noncancelation at end points and $(n+1)C(n)$ bounding the noncancelation at the (at most) $n+1$ transition points  between the affine functions.

\ 

\

Claims 1 and 2 and the symmetry of both $\mathcal{P}_n(m)$ and  $h_{n,m}(x)$ with respect to the $x_1$-axis (i.e. $i$-axis) imply that  

$$\hslash^0(A_n,m)=\iint_{{\mathcal{P}^+_n(m)}}h_{n,m}(x)dy+O(m)$$

\noindent  where $\mathcal{P}^+_n(m):=\mathcal{P}_n(m)\cap \{x_2\ge 0\}$. The polygon $\mathcal{P}^+_n(m)$ has a polygonal decomposition $\mathcal{P}^+_n(m)=\cup_{\ell=0}^{n+1}\mathcal{P}^{\ell}_n(m)$, where for each $\ell=0,...,n+1$,

$$h_{n,m,\ell}(x):=h_{n,m}(x)|_{\mathcal{P}^{\ell}_n(m)}$$ 

\noindent is given by a single  affine expression in $x=(x_1,x_2)$ (and $m$). The  polygons $\mathcal{P}^{\ell}_n(m)$, $\ell=0,...,n+1$  will be described via their vertices. To that end,  for $j=1,...,n$ consider the points:
    \begin{align*}
        v_j:= \left(
            \frac{2 j (m+1)}{j-n-3}+\frac{2 (j-1) (m+1)}{-j+n+2}+m, 
            \frac{2 (m+1)}{(j-n-3) (j-n-2)}
        \right)
    \end{align*}

\

	\begin{enumerate}
		\item[(i)] $\mathcal{P}^{0}_n(m)$:  $$\left(0,\frac{m}{n+1}\right)\to v_1\to \left(\frac{mn-2}{2+n},0\right)\to (0,0)\to \left(0,\frac{m}{n+1}\right),$$
  
		and where $h_{n,m,0}(x)=x_1+1.$
		\item[(ii)] $\mathcal{P}^{1}_n(m)$:
			\begin{align*}
				v_1\to v_2\to 
				\left(m,0\right)\to
				\left(\frac{mn-2}{2+n},0\right)\to v_1,
			\end{align*}
 
		and $h_{n,m,1}(x)=\frac{n(m-x_1)}{2}$.

		\item[(iii)] $\mathcal{P}^{j}_n(m)$, $j=2,\dots,n$: 
  
  For $j=2,\dots,n-1$,
  \begin{align*}
			v_j\to v_{j+1}\to (m,0)\to v_j 
			\end{align*}

   and for $j=n$,
   \begin{align*}
			v_n\to \left(n(m+1)-1,m+1\right) \to (m,0)\to v_n 
			\end{align*}
   
   and the function $h_{n,m,j}(x)= 
				\frac{1}{2} ((j-1)-n) (x_1-(j-1)x_2-m)$.

		\item[(iv)] $\mathcal{P}^{n+1}_n(m)$:
			\begin{align*}
				\left(0,\frac{m+2}{n+1}\right)&\to \left(n(m+1)-1,m+1\right)\to v_{n}\to \cdots\\ \cdots&\to v_2\to v_1\to \left(0,\frac{m}{n+1}\right)\to \left(0,\frac{m+2}{n+1}\right)
			\end{align*}
		 and where $h_{n,m,n+1}(x)=\frac{1}{2} (x_1-x_2 (n+1)+m+2)$.
	\end{enumerate}

\ 

The above $\mathcal{P}^{\ell}_n(m)$ and $h_{n,m,\ell}(x)$, $\ell=0,...,n+1$ give:

\begin{align*} \hslash^0(A_n,m)=\sum_{\ell=0}^{n+1} \iint_{{\mathcal{P}^{\ell}_n(m)}}h_{n,m,\ell}(x)dx=\hslash^0_\Omega(A_n)m^3+3\hslash^0_\Omega(A_n)m^2+O(m)
\end{align*}

\noindent with $\hslash^0_\Omega(A_n)$ as in (2.25).
\end{proof}

 \ 

 \ 

\begin{corollary}\label{cor:h0-an-pi-bounded} The invariants $\hslash^0_\Omega(A_n)$ increase with $n$ and are bounded. Moreover:

\begin{align*}
\lim_{n\to \infty} \hslash^0_\Omega(A_n)=\frac {2\pi^2}{9}-2
\end{align*}

\end{corollary}

\

\begin{remark}
 \begin{enumerate}
     \item [i)] The invariant $\hslash^0_\Omega(A_1)$ was known   \cite{thomas} (appeared  in \cite{bogomolov_nodes} unfortunately with an error). The $A_1$ case is quite direct since the exceptional locus has a single component. Some cases of $\hslash^0_\Omega(A_n)$ for low $n$ were known to the authors \cite{de2019resolutions} and \cite{weiss2020deformations}.
    \item [ii)] The function $\hslash^0(A_n,m)$ for fixed $n$  is a quasi-polynomial in $m$  of degree 3. This follows from the theory of polynomial weighted lattice sums over convex polytopes $\mathcal{P}({\bf b})$, ${\bf b}=(b_1,...,b_k)$, defined by $k$ inequalities, $\mu_l(x)\le b_l$, where the linear  forms $\mu_l(x)$ are fixed, but ${\bf b}$ varies (\cite{Ehrhart}, \cite{Brion}, \cite{Vergne}). A natural convex polygon decomposition of the polygons $\mathcal{P}_n^+(m)$ with the required properties  can be found  by decomposing $\mathcal{P}^{n+1}_n(m)$ into $n$ polygons by introducing vertical line segments above the points $v_2$,...,$v_n$.

     \item [iii)]  In \hyperref[thm:1]{Theorem 1}, we showed that the cubic and quadratic coefficients of $\hslash^0(A_n,m)$ have period 1.  In future work, we describe a divisibility condition for the  the least common multiple of the periods of the coefficients of the quasi-polynomial for all $n$, which allows us to determine $\hslash^0(A_n,m)$ for low $n$.  In the case of $A_1$ the least common multiple of the periods is 6 \cite{bruin2022explicit}. Knowing the functions $\hslash^0(A_n,m)$ can be used to obtain information on the degrees of symmetric differentials that occur on a surface which is a resolution of a surface with $A_n$ singularities. 
 \end{enumerate}

\end{remark}

\

\begin{example} For $A_2$ singularities the lcm of the periods of the coefficients is also $6$  and $\hslash^0(A_2,m)$ is described by the polynomials: 

\begin{align*}
    \hslash^0(A_2,m) =
\begin{cases}
 \frac{29}{216} m^3 + \frac{29}{72}m^2 + \frac{1}{12}m \text { }\text {  } &m\equiv 0\text { } (\text {mod} \text { } 6 )\\
 \frac{29}{216}m^3 + \frac{29}{72}m^2 + \frac{1}{8}m - \frac{143}{216}\text { }\text {  } &m\equiv 1\text { } (\text {mod} \text { } 6 )\\
 \frac{29}{216}m^3 + \frac{29}{72}m^2 + \frac{7}{36}m - \frac{2}{27} \text { }\text {  } &m\equiv 2\text { } (\text {mod} \text { } 6 )\\
 \frac{29}{216}m^3 + \frac{29}{72}m^2 + \frac{1}{8}m + \frac{3}{8}\text { }\text {  } &m\equiv 3\text { } (\text {mod} \text { } 6 )\\
 \frac{29}{216}m^3 + \frac{29}{72}m^2 + \frac{1}{12}m - \frac{10}{27}\text { }\text {  } &m\equiv 4\text { } (\text {mod} \text { } 6 )\\
 \frac{29}{216}m^3 + \frac{29}{72}m^2 + \frac{17}{72}m - \frac{7}{216} \text { }\text {  } &m\equiv 5\text { } (\text {mod} \text { } 6 )
\end{cases}
\end{align*}
\end{example}

\

\

\

\section{Applications}

\

\subsection{The $1$-st cohomological $\Omega$-asymptotics of $A_n$ singularities }

\ 

\

Let $X$ be the minimal resolution of an orbifold surface $Y$ of general type. The asymptotics of the localized component $Lh^1(X,S^m\Omega^1_X)$ of $h^1(X,S^m\Omega^1_X)$ described in (1.3) plays a role in the QS-bigness criterion (1.4). 

\ 

In this section, we establish the formula for the contribution given by each $A_n$ singularity to $Lh^1(X,S^m\Omega^1_X)$. This contribution   consists of:
\begin{align*}
     h^1_\Omega(A_n):=\lim_{m\to \infty}\frac{h^1(A_n,m)}{m^3}
\end{align*}
\noindent and it is called the $1-$cohomological $\Omega$-asymptotics  of $A_n$ (the limit exists).

\

\begin{proof} (of \hyperref[thm:2]{Theorem 2}) Using relation (1.7) for $A_n$ singularities, $$h^1(A_n,m)=
			\mu (A_n, m)- \chi_{\operatorname{orb}} (A_n, m)-\hslash^0(A_n,m)$$ we find $h^1_\Omega(A_n)$ to be

\begin{align*}
     h^1_\Omega(A_n)=\lim_{m\to \infty}\frac{\mu(A_n,m)}{m^3}-\lim_{m\to \infty}\frac{\chi_{\operatorname{orb}}(A_n,m)}{m^3}-\hslash^0_\Omega(A_n). 
\end{align*}

\ 
 
 The invariant $\hslash^0_\Omega(A_n)$ was determined in \hyperref[thm:1]{Theorem 1}. The invariants $\chi_\text{orb} (A_n,m)$ are given by formula (1.9) along with the local Chern  numbers
 
 \begin{align*} c_1^2(A_n) = 0
\hspace {.3in}\text {and}\hspace {.3in} c_2(A_n)=e(E) - \frac{1}{|G_{A_n}|}=\frac{n(n+2)}{n+1}, 
     \end{align*} 
where $e(E)$ is the topological Euler characteristic of the exceptional locus of the minimal resolution  and $|G_{A_n}|$ is the order of the local fundamental group of the $A_n$ singularity (\cite{blache} 3.18). Hence:

 \begin{align*} \lim_{m\to \infty}\frac{\chi_{\operatorname{orb}}(A_n,m)}{m^3}=\frac{s_2(A_n)}{3!}=-\frac{n(n+2)}{6(n+1)}.
    \end{align*}

 \

 Now, we turn our attention to  the invariants $\mu(A_n, m)$. The key feature of these invariants is that
 
 \begin{align*} \lim_{m\to \infty}\frac{\mu(A_n,m)}{m^3}=0
     \end{align*}

This is shown in \cite{blache} 4.4 or \cite {langer}) (it follows from general results on reflexive sheaves on quotient singularities, see also \cite{wahl_chernclasses}). A complete description of the invariants $\mu(A_n, m)$ is known to the authors and will appear in future work. For example, for fixed $n$, $\mu(A_n, m)$ is a quasi-polynomial of degree 1 in $m$.
\end{proof}

\ 

\begin{corollary}\label{cor:h1omega-An-goes-to-infty}  The invariants $h^1_\Omega(A_n)$ increase with $n$ and:

$$\lim_{n\to \infty} h^1_\Omega(A_n)=\infty$$
    
\end{corollary}

 \
 
 \begin{remark} The QS-bigness criterion (1.4) has $h^1_\Omega(A_n)$ as the contribution of $A_n$ to the localized component $Lh^1(X,S^m\Omega^1_X)$ of $h^1(X,S^m\Omega^1_X)$, while in the bigness criterion in \cite{roulleau2014} the contribution of $A_n$ to $h^1(X,S^m\Omega^1_X)$ is  $\frac{1}{2}\lim_{m\to \infty} \frac{\chi(A_n,m)}{m^3}$.  It follows from \hyperref[thm:1]{Theorem 1},  \hyperref[cor:h0-an-pi-bounded]{Corollary 2}, \hyperref[thm:2]{Theorem 2} and \hyperref[cor:h1omega-An-goes-to-infty]{Corollary 3} that the  contribution of each  $A_n$ in the QS-bigness criterion is always larger and approaches twice the contribution in the criterion of \cite{roulleau2014} as $n\to \infty$. In Part I of this work, we present the implications of this remark towards the wider range of pair of Chern numbers for which the CMS-bigness criterion can hold when compared to the criterion in \cite{roulleau2014}.
     
 \end{remark}

\

\begin{small}
\begin{table}[!htb]
\caption{ }
	\centering
	\resizebox{\columnwidth}{!}
 {\begin{tabular}{ p{1.5cm}  ccc ccc ccc c}
\hline
$n$& $1$ & $2$ & $3$& $4$ & $5$ & $6$ & $7$ \\ 
\hline
\addlinespace[1.5ex]
$h^1_\Omega(A_n)$  & $\dfrac{4}{27}$ & $\dfrac{67}{216}$ & $\dfrac{1283}{2700}$ & $\dfrac{577}{900}$ & $\dfrac{106819}{132300}$ & $\dfrac{1030727}{1058400}$ & $\dfrac{5431459}{4762800}$\\
\addlinespace[1.5ex]
\hline
\end{tabular}}

\end{table}
\end{small}

\

\subsection{ Extension of symmetric differentials for $A_n$ singularities}

\ 

\ 

%
%

Let $(\tilde {X},E)$ be the germ of the minimal resolution
of a quotient singularity $(X,x)$. It is well known that symmetric differentials $w\in H^0(\tilde {X}\setminus E,S^m\Omega^1_{\tilde {X}})$ acquire mild poles along the components of the exceptional divisor $E$. More precisely,  the poles of $w\in H^0(\tilde {X}\setminus E,S^m\Omega^1_{\tilde {X}})$ along $E$ are at worst logarithmic, i.e. $H^0(\tilde {X}\setminus E,S^m\Omega^1_{\tilde {X}})=H^0(\tilde {X},S^m\Omega^1_{\tilde {X}}(\text {log} E))$ (\cite{miyaoka1984maximal} 4.14, see also \cite{wahl_chernclasses} 4.7).

\ 

We show that for $A_n$ singularities, the poles that $w\in H^0(\tilde {X}\setminus E,S^m\Omega^1_{\tilde {X}})$ can acquire at the exceptional locus are milder (and to what extent) than logarithmic poles  (see \cite{thomas} for $A_1$). More precisely, in \hyperref[thm:3]{Theorem 3}, we give the maximal effective  divisor $D$ such that $$H^0(\tilde {X}\setminus E,S^m\Omega^1_{\tilde {X}})=H^0(\tilde {X},S^m\Omega^1_{\tilde {X}}(\text {log}E)\otimes \mathcal {O}_{\tilde X}(-D) )$$.

\subsubsection{Decomposition and the order of differentials in $S(\tilde {X}\setminus E)$}

In this section, $\tilde {X}$ is the minimal resolution of the affine model $X$ of the $A_n$ singularity as described in section 2.1. We have the commutative diagram involving the resolution $\sigma$ and the smoothing $\pi$ of $X$:

\begin{equation*}
\begin{tikzcd}
 & (\CC^2, 0) \arrow[dl, "\varphi"', dashed] \arrow[d, "\pi"] \\
(\tilde{X}, E) \arrow[r, "\sigma"]  & (X, x)  \\
\end{tikzcd}
\end{equation*}

The map $\varphi$ induces the isomorphisms between the algebras of symmetric differentials (see section 2.3.2):

\begin{align*} S\left(\tilde {X}\setminus E\right) &\xrightarrow[\cong] {\varphi^*} S\left(\CC^2\right)^{\mathbb{Z}_{n+1}}  \\ S\left(\tilde {X}\setminus \hat E\right)& \xrightarrow[\cong] {\varphi^*} S\left(\CC^*\times \CC^*\right)^{\mathbb{Z}_{n+1}}\tag {3.1}
    \end{align*}

\noindent (recall that $\hat E=E+E_0+E_{n+1}$ as in 2.1).

\ 

 Using the isomorphism (3.1) and the 3-gradation of $S\left(\CC^*\times \CC^*\right)^{\mathbb{Z}_{n+1}}$ described in 2.2.2,  we obtain the $(\hat k,i,m)-$decomposition of differentials $w \in S\left(\tilde {X}\setminus \hat E\right)$:

 \begin{align*}w=\sum_{\substack{m\in \mathbb{Z}_{\ge  0}\\ i,\hat k\in \mathbb{Z}}}w_{\hat k,i,m} \tag {3.2}\end{align*}

\noindent where the $w_{\hat k,i,m}\in H^0(\tilde {X}\setminus \hat E,S^m\Omega^1_{\tilde {X}})$ are such that  $\varphi^*w_{\hat k,i,m}\in \hat V_{\hat k,i,m}$. If $w \in S\left(\tilde {X}\setminus E\right)$, then we have

\begin{align*}w=\sum_{\substack{m, i\in \mathbb{Z}_{\ge  0}\\|\hat k|\le \frac{i+m}{n+1}}}w_{\hat k,i,m}. \tag {3.3}\end{align*}

\noindent In this case, we additionally have that $\varphi^*w\in S(\CC^2) $ and hence $\varphi^*w_{\hat k,i,m}\in \hat V^{\text{reg}}_{\hat k,i,m}$. 

We  say a differential $w\in  S\left(\tilde {X}\setminus \hat E\right)$  is of $type$ $(\hat k,i,m)$ if $\varphi^*w\in \hat V_{\hat k,i,m}$. The expression (3.2) is the decomposition of $w$ relative to the $(\hat k,i,m)$-types.

\

We define the $order$ of a symmetric differential $w\in  S\left(\tilde {X}\setminus \hat E\right)$ to be:

\begin{align*}\text {ord}(w)=\min \left\{i \text { } \Big| \text { }\varphi^*w= \sum_{\substack{ i,\hat k\in \mathbb{Z}}}w_{\hat k,i,m} \text  { and } w_{\hat k,i,m}\neq 0\right\}, \tag {3.4} \end{align*}

\noindent i.e. $\operatorname{ord}(w)$ is the smallest order of the differential monomials appearing in the monomial decomposition of $\varphi^*w$. The ord$(w)$ can also be described as the order of vanishing of $\varphi^*w$ at $0\in \CC^2$, this is consistent with the definition of order in theorem B for $w\in H^0(X\setminus x,S^m\Omega^1_X)$.

\ 

\subsubsection{Comparison with logarithmic poles}

\

We  characterize the allowed poles of a differential $w\in H^0(\tilde {X}\setminus  E,S^m\Omega^1_{\tilde {X}})$ along $E$  via a comparison to the  maximum poles allowed on logarithmic symmetric differentials $\mu \in H^0(\tilde X\setminus E,S^m\Omega^1_{\tilde  X}(\log E))$.

\

\

\begin{proof}(of \hyperref[thm:3]{Theorem 3}(a)) Let $w\in H^0(\tilde X\setminus E,S^m\Omega^1_{\tilde  X})$. The worst pole of $w$ along a component $E_r$ of $E$ will be the worst pole attained  by one of its $(\hat k,i,m)$-type components in the  decomposition (3.3).

\

Let $w$ be of $(\hat k,i,m)$-type. To determine the poles of $w$ along the components of $E$, we examine $w$ on the coordinate patches $U_r$, $r=0,...,n$, (described in 2.1) covering  $\tilde X$. Set $w_r:=w|_{U_r}$, note  $w_r=(\varphi_r^*)^{-1}(\varphi^*w)$ hence in $(\varphi_r^*)^{-1}\hat V_{\hat k,i,m}$. Using \ref{eq:2.17}, it follows that $w_r$ is in the span of the monomials in the block $B_{\frac {i+m}{2}+(\frac{n+1}{2}-r)\hat{k}, i+(n-2r)\hat{k}, m}$.

\

The relevant  observation  towards the comparison of the poles of $w\in H^0(\tilde X\setminus E,S^m\Omega^1_{\tilde  X})$ to logarithmic poles is the fact that we can rewrite the monomials in a block $B_{k,i,m}$, in the following form:
 
$$B_{k,i,m}=\Big \{z_1^{i+m-k}z_2^{k}\frac {dz_1^{m-q}}{z_1^{m-q}}\frac{dz_2^q}{z_2^q}\Big \}_{q=0,...,m}$$

\ 

Hence $w_r$ (of type $(\hat k,i,m)$) is given by a sum of the following monomials:

\begin{align*}\hspace{.4in} u_1^{\frac{i+m}{2}+(\frac{n-1}{2}-r)\hat k}u_2^{\frac{i+m}{2}+(\frac{n+1}{2}-r)\hat k}\frac {du_1^{m-q}}{u_1^{m-q}}\frac{du_2^q}{u_2^q}, \hspace{.2in}q=0,...,m\tag {3.5} \end{align*}

\

Next we prove the result for $A_n$ with $n\ge 2$ (the case $n=1$ follows from the same argument after a minor setup adjustment). For $r=1,...,n-1$:

$$S^m\Omega^1_{\tilde X}(\log E)|_{U_r}=\bigoplus_{q=0}^m\mathcal{O}_{U_r}\frac {du_1^{m-q}}{u_1^{m-q}}\frac{du_2^q}{u_2^q}$$

\noindent (recall: for each $r=1,...,n-1$, $E_r\cap U_r=\{u_2=0\}$ and $E_{r+1}\cap U_r=\{u_1=0\}$).

\

It follows from (3.5) that the order of the poles along $E_r$ of the monomials  involved in the sum giving $w_r$  deviates from the highest pole order possible for a logarithmic symmetric differential by subtracting $\frac{i+m}{2}+(\frac{n+1}{2}-r)\hat k$. Hence:

\begin{align*}w_{r}\in H^0\left (U_r,S^m\Omega^1_{\tilde X}(\log E)\otimes \mathcal{O}(-D)\right ) \end{align*}

\noindent with 

\vspace{-.15in}

\begin{align*}D=\left [\frac{i+m}{2}+\left(\frac{n+1}{2}-r\right)\hat k\right ]E_r+\left[\frac{i+m}{2}+\left(\frac{n-1}{2}-r\right)\hat k\right ]E_{r+1}\end{align*}

\

An easy consequence of the above is:

$$H^0(\tilde X\setminus E,S^m\Omega^1_{\tilde  X})=H^0(\tilde X,S^m\Omega^1_{\tilde  X}(\log E))$$

\noindent This follows since the condition that $D\ge 0$ in all $r=1,...,n-1$ is equivalent to $|\hat k|\le \frac{i+m}{n-1}$, which holds since the $(\hat k,i,m)$-components of $w\in H^0(\tilde X\setminus E,S^m\Omega^1_{\tilde  X})$ satisfy $|\hat k|\le \frac{i+m}{n+1}$.

\

\

The full strength of the result follows from:

$$a_r=\min \begin{cases}\frac{i+m}{2}+(\frac{n+1}{2}-r)\hat k \text { }\big | \text { } i\ge 0, |\hat k|\le \frac {i+m}{n+1}, (n+1)\hat k\equiv i+m \bmod{2}\Big \}\end{cases}$$

\noindent satisfies:

$$a_r:= \begin{cases}\sum_{j=1}^r \ceil {\frac{m+2-2j}{n+1}}& 1\le r\le\floor {\frac{n+1}{2}}\\\sum_{j=1}^{n+1-r} \ceil {\frac{m+2-2j}{n+1}}&\floor {\frac{n+1}{2}}<r \le n\end{cases}$$

\

\

 Set  $D=\sum_{r=1}^na_rE_r$. By construction all the $(\hat k,i,m)$-components of $w$, and hence $w$, belong to  $H^0(\tilde {X},S^m\Omega^1_{\tilde {X}}(\log E)\otimes \mathcal {O}_{\tilde X}(-D))$ and the maximality of $D$ is guaranteed by the definition  of $a_r$ and the fact.
    
\end{proof}

\

\subsubsection{ Order and holomorphic extension}

\ 

\ 

Part (b) of  \hyperref[thm:3]{Theorem 3} gives a criterion for the holomorphic extension  to $\tilde X$ of symmetric differentials on $\tilde X\setminus E$ involving the order of the differentials. This result is used through out section 2, it makes explicit that the $h^0(A_n,m)$ are finite since it implies that the polygon $\mathcal{P}_n(m)$ over which the lattice weighted sum giving $h^0(A_n,m)$ is bounded. We note that  this result can also be derived from corollary 1.

\begin{proof} (of \hyperref[thm:3]{Theorem 3} part (b)) If $w\in H^0(\tilde {X}\setminus E,S^m\Omega^1_{\tilde {X}})$, then the $(\hat k,i,m)$-decomposition of $w$ has $w=\sum w_{\hat k,i,m}$, with $i\in \mathbb{Z}_{\ge  0}$ and $|\hat k|\le \frac{i+m}{n+1}$. The differential $w$ extends regularly to $\tilde X$ if all the $w_{\hat k,i,m}$ do. Note that each $w_{\hat k,i,m}\in H^0(\tilde {X}\setminus E,S^m\Omega^1_{\tilde {X}})$, hence no need to check about poles along $E_0$ or $E_{n+1}$.

\

 The  conditions to guarantee no poles of $w_{\hat k,i,m}$ along $E_r$, $r=1,...,n$ are obtained using (3.5). From (3.5) it follows that all $w$ of type $(\hat k,i,m)$ do not acquire a pole along $E_r$  if and only if:

\begin{align*}
 \frac {i+m}{2}+\left(\frac{n+1}{2}-r\right)\hat k\ge m \tag {3.6}
\end{align*}

First observation is that the $r$ conditions (3.6) on the triples $(\hat k,i,m)$ can be reduced to the conditions for $r=1$ and $r=n$. Second observation is that for $i<m$ the conditions can't hold. Finally, if $i\ge m$, the conditions become:

\begin{align*}
 |\hat k|\le \frac {i-m}{n-1}
\end{align*}

The  condition above holds for the possible $\hat k$, i.e. in the range given by $|\hat k|\le \frac{i+m}{n+1}$, if $i\ge nm$.

\end{proof}

\ 

\

\vspace {-.2in}

\bibliographystyle{amsalpha}
\bibliography{references.bib}

\providecommand{\bysame}{\leavevmode\hbox to3em{\hrulefill}\thinspace}
\providecommand{\MR}{\relax\ifhmode\unskip\space\fi MR }
\providecommand{\MRhref}[2]{%
  \href{http://www.ams.org/mathscinet-getitem?mr=#1}{#2}
}
\providecommand{\href}[2]{#2}
\begin{thebibliography}{AdOW23}

\bibitem[AdOW23]{ADOWI}
Y.~D. Asega, B.~de~Oliveira, and M.~Weiss, \emph{Surface quotient singularities
  and bigness of the cotangent bundle: Part {$\text {I}$}}, 2023.

\bibitem[BDO06]{bogomolov_nodes}
F.~Bogomolov and B.~De~Oliveira, \emph{Hyperbolicity of nodal hypersurfaces},
  J. Reine Angew. Math. \textbf{596} (2006), 89--101.

\bibitem[Bla96]{blache}
R.~Blache, \emph{Chern classes and {H}irzebruch-{R}iemann-{R}och theorem for
  coherent sheaves on complex-projective orbifolds with isolated
  singularities}, Math. Z. \textbf{222} (1996), no.~1, 7--57.

\bibitem[Bog78]{bogomolov_stability}
F.~Bogomolov, \emph{Holomorphic tensors and vector bundles on projective
  manifolds}, Izv. Akad. Nauk SSSR Ser. Mat. \textbf{42} (1978), no.~6,
  1227--1287, 1439.

\bibitem[BTVA22]{bruin2022explicit}
Nils Bruin, Jordan Thomas, and Anthony V{\'a}rilly-Alvarado, \emph{Explicit
  computation of symmetric differentials and its application to
  quasihyperbolicity}, Algebra \& Number Theory \textbf{16} (2022), no.~6,
  1377--1405.

\bibitem[BV97]{Brion}
Michel Brion and Mich\`ele Vergne, \emph{Residue formulae, vector partition
  functions and lattice points in rational polytopes}, J. Amer. Math. Soc.
  \textbf{10} (1997), no.~4, 797--833. \MR{1446364}

\bibitem[BV12]{Vergne}
Nicole Berline and Mich\`ele Vergne, \emph{Analytic continuation of a
  parametric polytope and wall-crossing}, Configuration spaces, CRM Series,
  vol.~14, Ed. Norm., Pisa, 2012, pp.~111--172. \MR{3203637}

\bibitem[DOW19]{de2019resolutions}
Bruno De~Oliveira and Michael~L Weiss, \emph{Resolutions of surfaces with big
  cotangent bundle and {$ A_2 $} singularities}, arXiv preprint
  arXiv:1912.09566 (2019).

\bibitem[Ehr62]{Ehrhart}
Eug\`ene Ehrhart, \emph{Sur les poly\`edres rationnels homoth\'{e}tiques \`a
  {$n$} dimensions}, C. R. Acad. Sci. Paris \textbf{254} (1962), 616--618.
  \MR{130860}

\bibitem[Lan00]{langer}
A.~Langer, \emph{Chern classes of reflexive sheaves on normal surfaces}, Math.
  Z \textbf{235} (2000), 591–--614.

\bibitem[Miy84]{miyaoka1984maximal}
Y.~Miyaoka, \emph{The maximal number of quotient singularities on surfaces with
  given numerical invariants}, Mathematische Annalen \textbf{268} (1984),
  no.~2, 159--171.

\bibitem[RR14]{roulleau2014}
X.~Roulleau and E.~Rousseau, \emph{Canonical surfaces with big cotangent
  bundle}, Duke Math. J. \textbf{163} (2014), no.~7, 1337--1351.

\bibitem[Tho13]{thomas}
J.~Thomas, \emph{{Contraction Techniques in the Hyperbolicity of Hypersurfaces
  of General Type}}, Ph.D. thesis, New York University, 2013.

\bibitem[Wah93]{wahl_chernclasses}
J.~Wahl, \emph{Second {C}hern class and {R}iemann-{R}och for vector bundles on
  resolutions of surface singularities}, Math. Ann. \textbf{295} (1993), no.~1,
  81--110.

\bibitem[Wei20]{weiss2020deformations}
Michael Weiss, \emph{Deformations of smooth hypersurfaces in {$\mathbb{P}^3$}
  with big cotangent bundle}, Ph.D. thesis, University of Miami, 2020.

\end{thebibliography}

\end{document}